\documentclass[reqno]{amsart}

\usepackage[utf8]{inputenc}
\usepackage[english]{babel}

\usepackage{braket}
\usepackage{amsmath}
\usepackage{amstext}
\usepackage{amssymb}
\usepackage{amsthm}
\usepackage{mathtools}
\usepackage{stmaryrd}
\usepackage{mathrsfs}
\usepackage{extarrows}
\usepackage{afterpage}
\usepackage{tabu}
\usepackage{array}

\usepackage{microtype}

\usepackage{graphicx}
\usepackage{amsfonts}

\usepackage{units}
\usepackage{enumerate}
\usepackage{url}
\usepackage{multirow}

\newcolumntype{C}[1]{>{\centering\let\newline\\\arraybackslash\hspace{0pt}}m{#1}}

\usepackage{IEEEtrantools}

\usepackage{tikz}
\usetikzlibrary{matrix,arrows,calc,fit}

\usepackage{hyperref}

\usepackage[margin=1cm]{caption}
\usepackage{colortbl}
\newcommand{\graycell}{\cellcolor{black!10}}

\allowdisplaybreaks[4]

\theoremstyle{definition}
\newtheorem{theorem}{Theorem}[section]

\newtheorem{remark}[theorem]{Remark}

\newtheorem{lemma}[theorem]{Lemma}

\newtheorem{corollary}[theorem]{Corollary}

\newtheorem{matching}[theorem]{Matching}

\newenvironment{breakmatching}[1][]{\begin{matching}[#1]\ \\[-0.4cm]}{\end{matching}}

\newenvironment{provedcorollary}
  {\pushQED{\qed}\begin{corollary}}
  {\popQED\end{corollary}}

\newtheorem{innercustomthm}{Theorem}
\newenvironment{customthm}[1]
  {\renewcommand\theinnercustomthm{#1}\innercustomthm}
  {\endinnercustomthm}

\newtheorem{innercustomcor}{Corollary}
\newenvironment{customcor}[1]
  {\renewcommand\theinnercustomcor{#1}\innercustomcor}
  {\endinnercustomthm}

\newcommand{\Z}{\mathbb{Z}}

\newcommand{\Q}{\mathbb{Q}}

\newcommand{\M}{\mathcal{M}}

\newcommand{\W}{{\bf W}}
\newcommand{\G}[1]{\bf G_{#1}}
\newcommand{\GW}{{\bf G_W}}
\newcommand{\GWn}[1]{\bf G_{W_{#1}}}
\newcommand{\FW}{{\bf F_W}}

\newcommand{\de}{\partial}

\DeclareMathOperator{\rk}{rk}

\newcommand{\XW}{\mathbf X_{\mathbf W}}

\newcommand{\xor}{\veebar}

\renewcommand{\pmod}[2][]{\;#1(\text{mod}\; #2#1)}

\title{On the local homology of Artin groups of finite and affine type}
\author{Giovanni Paolini}

\begin{document}

\begin{abstract}
  We study the local homology of Artin groups using weighted discrete Morse theory.
  In all finite and affine cases, we are able to construct Morse matchings of a special type (we call them ``precise matchings'').
  The existence of precise matchings implies that the homology has a square-free torsion.
  This property was known for Artin groups of finite type, but not in general for Artin groups of affine type.
  We also use the constructed matchings to compute the local homology in all exceptional cases, correcting some results in the literature.
\end{abstract}

\maketitle

\setcounter{tocdepth}{1}
\tableofcontents

\section{Introduction}

In this paper we study the local homology of Artin groups with coefficients in the Laurent polynomial ring $R = \Q[q^{\pm 1}]$, where each standard generator acts as a multiplication by $-q$.
This homology has been already thoroughly investigated for groups of finite type \cite{frenkel1988cohomology, de1999arithmetic, de2001arithmetic, callegaro2005cohomology, salvetti2015some, paolini2017weighted}, also with integral coefficients \cite{callegaro2004integral, callegaro2006homology}, and for some groups of affine type \cite{callegaro2008cohomology, callegaro2008cohomology2, callegaro2010k, salvetti2013combinatorial, paolini2017weighted}.
This work is meant to be a natural continuation of \cite{paolini2017weighted}, and is based on the combinatorial techniques developed in \cite{salvetti2013combinatorial, paolini2017weighted}.

In \cite{salvetti2013combinatorial} Salvetti and Villa introduced a new combinatorial method to study the homology of Artin groups, based on discrete Morse theory.
They described the general framework, and made explicit computations for all exceptional affine groups.
In \cite{paolini2017weighted} the author and Salvetti developed the theory further, and made explicit computations for the affine family $\smash{\tilde C_n}$ as well as for the (already known) families $A_n$, $B_n$ and $\smash{\tilde A_n}$.
A common ingredient emerged in all the cases considered there, namely \emph{precise matchings}.
It was shown that whenever an Artin group admits precise matchings, then its local homology has a square-free torsion.
For Artin groups of finite type, the absence of higher powers in the torsion is a consequence of the isomorphism with the (constant) homology of the corresponding Milnor fiber \cite{callegaro2005cohomology}.
This geometric argument does not apply to Artin groups of infinite type, but precise matchings proved to be useful also beyond the finite case.

In the present work we show that precise matchings exist for all Artin groups of finite and affine type.
As we said, this implies the square-freeness of the torsion in the local homology.
The elegance of this conclusion seems to hint at some unknown deeper geometric reason.
The main results, stated in Section \ref{sec:main-theorem}, are the following.

\begin{customthm}{\ref{thm:main-theorem}}
  Every Artin group of finite or affine type admits a $\varphi$-precise matching for each cyclotomic polynomial $\varphi$.
\end{customthm}

\begin{customcor}{\ref{cor:main-corollary}}
  Let $\GW$ be an Artin group of finite or affine type.
  Then the local homology $H_*(\XW; R)$ has no $\varphi^k$-torsion for $k\geq 2$.
\end{customcor}

We are able to use precise matchings to carry out explicit homology computations for all exceptional finite and affine cases. In particular we recover the results of \cite{de1999arithmetic,salvetti2013combinatorial}, with small corrections.
The matchings we find for $D_n$, $\smash{\tilde B_n}$, and $\smash{\tilde D_n}$ are quite complicated, so we prefer to omit explicit homology computations for these cases (the homology for $D_n$ and $\smash{\tilde B_n}$ was already computed in \cite{de1999arithmetic} and \cite{callegaro2010k}, respectively).
The remaining finite and affine cases, namely $A_n$, $B_n$, $\smash{\tilde A_n}$ and $\smash{\tilde C_n}$, were already discussed in \cite{paolini2017weighted}.

We also provide a software library which can be used to generate matchings for any finite or affine Artin group, check preciseness, and compute the homology.
Source code and instructions are available online \cite{precise-matchings}.

This paper is structured as follows.
In Section \ref{sec:framework} we review the general combinatorial framework developed in \cite{salvetti2013combinatorial, paolini2017weighted}. We introduce the local homology $H_*(\XW; R)$, which is the object of our study, together with algebraic complexes to compute it. We present weighted discrete Morse theory and precise matchings, and recall some useful results.
In Section \ref{sec:main-theorem} we state and discuss the main results of this paper. Subsequent sections are devoted to the proof of the main theorem.
In Section \ref{sec:reduction-to-irreducible} we show that it is enough to construct precise matchings for irreducible Artin groups.
In Section \ref{sec:weights} we recall the computation of the weight of irreducible components of type $A_n$, $B_n$ and $D_n$, which is used later.
In Sections \ref{sec:An}-\ref{sec:I2} we construct precise matchings for the families $A_n$, $D_n$, $\smash{\tilde B_n}$, $\smash{\tilde D_n}$ and $I_2(m)$.
Finally, in Section \ref{sec:exceptional-cases} we deal with the exceptional cases.

\section{Local homology of Artin groups via discrete Morse theory}
\label{sec:framework}

In this section we are going to recall the general framework of \cite{salvetti2013combinatorial, paolini2017weighted} for the computation of the local homology $H_*(\XW; R)$.

Let $(\W, S)$ be a Coxeter system on a finite generating set $S$, and let $\Gamma$ be the corresponding Coxeter graph (with $S$ as its vertex set).
Denote by $\GW$ the corresponding Artin group, with standard generating set $\Sigma = \{ g_s \mid s \in S\}$.
Define $K_\W$ as the (finite) simplicial complex over $S$ given by
\[ K_\W = \{ \sigma \subseteq S \mid \text{the parabolic subgroup $\W_{\!\sigma}$ generated by $\sigma$ is finite} \}. \]
It is convenient to include the empty set $\varnothing$ in $K_\W$.
Let $\XW$ be the quotient of the Salvetti complex of $\W$ by the action of $\W$. This is a finite (non-regular) CW complex, with polyhedral cells indexed by $K_\W$.
It has $\GW$ as its fundamental group, and it is conjectured to be a space of type $K(\GW, 1)$ \cite{salvetti1987topology, salvetti1994homotopy, paris2012k}.
This conjecture is known to be true for all groups of finite type \cite{deligne1972immeubles} and for some families of groups of infinite type, including the affine groups of type $\smash{\tilde A_n}$, $\smash{\tilde B_n}$, and $\smash{\tilde C_n}$ \cite{okonek1979dask, callegaro2010k}.

Consider the action of the Artin group $\GW$ on the ring $R = \Q[q^{\pm 1}]$ given by
\[ g_s\ \mapsto \  [\text{multiplication by } -q]\quad \forall s\in S. \]
We are interested in studying the local homology $H_*(\XW; R)$, with coefficients in the local system defined by the above action of $\GW = \pi_1(\XW)$ on $R$.
Whenever $\XW$ is a $K(\GW, 1)$ space, this coincides with the group homology $H_*(\GW; R)$ with coefficients in the same representation.

The local homology $H_*(\XW; R)$ is computed by the algebraic complex
\[ C_k = \, \bigoplus_{\mathclap{\substack{\sigma \in K_\W\\
      |\sigma|=k}}}
    \, R\, e_\sigma \]
with boundary
\[  \partial(e_\sigma) \, = \, \sum_{\mathclap{\substack{
  \tau \lhd \sigma }}}
  \; [\sigma:\tau]\, \frac{\W_{\!\sigma}(q)}{\W_{\!\tau}(q)}\, e_\tau, \]
where $\W_{\!\sigma}(q)$ is the Poincaré polynomial of the parabolic subgroup $\W_{\!\sigma}$ of $\W$.

Let $C^0_*$ be the $1$-shifted algebraic complex of free $R$-modules which computes the reduced simplicial homology of $K_\W$ with (constant) coefficients in $R$. Namely:
\[ C^0_k = \, \bigoplus_{\mathclap{\substack{\sigma \in K_\W\\
      |\sigma|=k}}}
    \, R\, e^0_\sigma \]
with boundary
\[  \partial^0(e^0_\sigma) \, = \, \sum_{\mathclap{\substack{
  \tau \lhd \sigma }}}
  \; [\sigma:\tau]\, e^0_\tau. \]
Then we have an injective chain map $\Delta\colon C_* \to C^0_*$ defined by
\[ e_\sigma \mapsto \W_{\!\sigma}(q) \, e^0_\sigma. \]
Therefore there is an exact sequence of complexes:
\[ 0 \longrightarrow C_* \xlongrightarrow{\Delta} C^0_* \xlongrightarrow{\pi} L_* \longrightarrow 0, \]
where
\[ L_k \, = \, \bigoplus_{\mathclap{\substack{
    \sigma \in K_\W \\
    |\sigma| = k}}}
  \; \frac{R}{(\W_{\!\sigma}(q))}\; \bar{e}_\sigma, \]
with boundary induced by the boundary of $C^0_*$.
The associated long exact sequence in homology then allows to compute the homology of $C_*$ in terms of the homology of $C^0_*$ and of $L_*$:
\begin{equation*}
  \dots \xrightarrow{\pi_*} H_{k+1}(L_*) \xrightarrow{\delta} H_k(C_*) \xrightarrow{\Delta_*} H_k(C_*^0) \xrightarrow{\pi_*} H_k(L_*) \xrightarrow{\delta} H_{k-1}(C_*) \xrightarrow{\Delta_*} \dots
\end{equation*}

In this paper we mostly focus on Artin groups of finite and affine type, for which $K_\W$ is either the full simplex (in the finite case) or its boundary (in the affine case). In the former case $H_*(C^0_*)$ is trivial, and in the latter case the only non-trivial term is $H_{|S|-1}(C^0_*) \cong R$.
Therefore the challenging part consists in understanding the homology of the complex $L_*$, which encodes all the torsion.

Poincaré polynomials of Coxeter groups are products of cyclotomic polynomials $\varphi_d$ with $d\geq 2$. Thus the complex $L_*$ decomposes as a direct sum of $\varphi$-primary components $(L_*)_\varphi$, where $\varphi=\varphi_d$ varies among the cyclotomic polynomials.
Each component $(L_*)_\varphi$ takes the following form:
\[ (L_k)_\varphi \, = \, \bigoplus_{\mathclap{\substack{
    \sigma \in K_\W \\
    |\sigma| = k}}}
  \; \frac{R}{(\varphi^{v_\varphi(\sigma)})}\; \bar{e}_\sigma \]
where $v_\varphi(\sigma)$ is the multiplicity of $\varphi$ in the factorization of $\W_{\!\sigma}(q)$.
Again, the boundary in $(L_*)_\varphi$ is induced by the boundary of $C^0_*$.
The homology of $L_*$ decomposes accordingly, so our goal is to study $H_*((L_*)_\varphi)$ for each cyclotomic polynomial $\varphi$.

The main tool we are going to use is algebraic Morse theory for weighted complexes.
We recall here the main points of the theory, referring to \cite{forman1998morse, chari2000discrete, batzies2002discrete, kozlov2008combinatorial, salvetti2013combinatorial} for more details.
Let $G$ be the incidence graph of $K_\W$, i.e.\ the graph having $K_\W$ as its vertex set and with a directed edge $\sigma \to \tau$ whenever $\tau \lhd \sigma$.
We still call \emph{simplices} the vertices of $G$, to avoid confusion with the vertices of $K_\W$.
A matching $\M$ on $G$ is a set of edges of $G$ such that each simplex $\sigma \in K_\W$ is adjacent to at most one edge in $\M$.
We say that $\sigma$ is \emph{critical} (with respect to the matching $\M$) if none of its adjacent edges is in $\M$.
We also call \emph{alternating path} a sequence
\begin{equation*}
  (\tau_0 \,\lhd) \, \sigma_0\rhd \tau_1\lhd \sigma_1 \rhd \tau_2 \lhd \sigma_2 \rhd \dots \rhd \tau_{m} \lhd \sigma_m \,(\rhd\, \tau_{m+1})
\end{equation*}
such that each pair $\sigma_i \lhd \tau_i$ belongs to $\M$ and no pair $\sigma_i \lhd \tau_{i-1}$ belongs to $\M$.
An \emph{alternating cycle} is a closed alternating path.
The following are key definitions of the theory.
\begin{itemize}
 \item $\M$ is \emph{acyclic} if all alternating cycles are trivial.
 \item $\M$ is $\varphi$-\emph{weighted} if $v_\varphi(\sigma) = v_\varphi(\tau)$ whenever $(\sigma \to \tau) \in \M$.
\end{itemize}
Notice that in general, by construction, one has $v_\varphi(\sigma) \geq v_\varphi(\tau)$ for $\tau \lhd \sigma$.

\begin{theorem}[{\cite[Theorem 2]{salvetti2013combinatorial}}]
  Fix a cyclotomic polynomial $\varphi$. Let $\M$ be an acyclic and $\varphi$-weighted matching on $G$.
  Then the homology of $(L_*)_\varphi$ is the same as the homology of the \emph{Morse complex}
  \[ (L_*)_\varphi^\M \, = \bigoplus_{\sigma \text{ critical} } \frac{R}{(\varphi^{v_\varphi(\sigma)})} \ \bar{e}_{\sigma} \]
  with boundary
  \[ \de^\M(\bar e_\sigma) = \, \sum_{\mathclap{\substack{
    \tau \text{ critical}\\
    |\tau| \,=\, |\sigma| - 1}}} \; [\sigma:\tau]^\M \, \bar e_\tau, \]
  where $[\sigma:\tau]^\M \in \Z$ is given by the sum over all alternating paths
  \[ \sigma \vartriangleright \tau_1 \vartriangleleft \sigma_1 \vartriangleright \tau_2 \vartriangleleft \sigma_2 \vartriangleright \dots \vartriangleright \tau_m \vartriangleleft \sigma_m \vartriangleright \tau \]
  from $\sigma$ to $\tau$ of the quantity
  \[ (-1)^m [\sigma:\tau_1][\sigma_1:\tau_1][\sigma_1:\tau_2][\sigma_2:\tau_2] \cdots [\sigma_m:\tau_m][\sigma_m:\tau]. \]
\end{theorem}

In \cite{paolini2017weighted} a special class of weighted matchings was introduced, namely precise matchings.
We say that $\M$ is $\varphi$-\emph{precise} (or simply \emph{precise}) if $\M$ is acyclic and $\varphi$-weighted, and has the following additional property: $v_\varphi(\sigma) = v_\varphi(\tau) + 1$ whenever $[\sigma : \tau]^\M \neq 0$ (here $\sigma$ and $\tau$ are critical simplices, so that $[\sigma:\tau]^\M$ is defined).
This condition can be thought as a rigid (weight-consistent) maximality condition.
It appears to arise naturally in the study of the local homology of Artin groups, as shown in \cite{paolini2017weighted} and in the present work.

We refer to \cite[Section 4]{paolini2017weighted} for a general introduction to precise matchings.
Here we are only going to briefly recall what we need to study Artin groups.
Let $C_*^0(\Q)$ be the $1$-shifted algebraic complex of free $\Q$-modules which computes the reduced simplicial homology of $K_\W$ with coefficients in $\Q$ (this is the same as $C^0_*$, but with $\Q$ instead of $R$ in the definition). An acyclic matching $\M$ on $G$ can be used to compute a Morse complex $C_*^0(\Q)^\M$ of $C_*^0(\Q)$ as well.
Call $\delta_*^\M$ the boundary of the Morse complex $C_*^0(\Q)^\M$.

\begin{theorem}[{\cite[Theorem 5.1]{paolini2017weighted}}]
  Fix a cyclotomic polynomial $\varphi$. Let $\M$ be a $\varphi$-precise matching on $G$.
  Then the $\varphi$-torsion component of the local homology $H_*(\XW; R)$ in each dimension, as an $R$-module, is given by
  \[ H_m(\XW; R)_\varphi \cong \left( \frac{R}{(\varphi)} \right)^{\oplus \rk \delta_{m+1}^\M}. \]
\end{theorem}

\begin{theorem}[{\cite[Theorem 5.1]{paolini2017weighted}}]
  Suppose to have a $\varphi$-precise matching $\M_\varphi$ on $G$ for every cyclotomic polynomial $\varphi$.
  Then the local homology of $\XW$ in each dimension, as an $R$-module, is given by
  \[ H_m(\XW; R) \cong \left( \bigoplus_\varphi \left( \frac{R}{(\varphi)} \right)^{\oplus \rk \delta_{m+1}^{\M_\varphi}} \right) \oplus H_m(C^0_*). \]
  In particular the term $H_*(C^0_*)$ gives the free part of the homology, whereas the other direct summands give the torsion part.
  \label{thm:homology-artin-groups}
\end{theorem}

\begin{corollary}[{\cite[Corollary 5.2]{paolini2017weighted}}]
  Suppose that $\GW$ is an Artin group that admits a $\varphi$-precise matching for every cyclotomic polynomial $\varphi$. Then the homology $H_*(\XW;R)$ has no $\varphi^k$-torsion for $k\geq 2$.
  \label{cor:squarefree-torsion-precise}
\end{corollary}

The formula of Theorem \ref{thm:homology-artin-groups} simplifies further when $\GW$ is of finite type (the free part disappears) or of affine type (the free part only appears in dimension $|S|-1$ and has rank $1$).

\section{The main theorem}
\label{sec:main-theorem}

As mentioned in the introduction, the main result of this paper is the following.

\begin{theorem}
  Every Artin group of finite or affine type admits a $\varphi$-precise matching for each cyclotomic polynomial $\varphi$.
  \label{thm:main-theorem}
\end{theorem}

By Corollary \ref{cor:squarefree-torsion-precise}, this has the following immediate consequence.

\begin{provedcorollary}
  Let $\GW$ be an Artin group of finite or affine type.
  Then the local homology $H_*(\XW; R)$ has no $\varphi^k$-torsion for $k\geq 2$.
  \label{cor:main-corollary}
\end{provedcorollary}

For Artin groups of finite type, the local homology $H_*(\XW; R)$ coincides with the (constant) homology $H_*(\FW; \Q)$ of the Milnor fiber of the associated hyperplane arrangement \cite{callegaro2005cohomology}.
The $q$-multiplication on the homology of $\XW$ corresponds to the action of the monodromy operator on the homology of $\FW$.
The monodromy operator has a finite order $N$, thus the polynomial $q^N - 1$ must annihilate the homology. Therefore there can only be square-free torsion.

The fact that the same conclusion holds for Artin groups of affine type is surprising, and might be due to some deeper geometric reasons which we still do not know.

The proof of Theorem \ref{thm:main-theorem} is split throughout the rest of this paper.
In Section \ref{sec:reduction-to-irreducible} we show that it is enough to construct precise matchings in the irreducible finite and affine cases.
Case $A_n$ was done in \cite{paolini2017weighted}. However, we study it again in Section \ref{sec:An} as we need it for $D_n$, $\smash{\tilde B_n}$ and $\smash{\tilde D_n}$.
Cases $B_n$, $\smash{\tilde A_n}$ and $\smash{\tilde C_n}$ were done as well in \cite{paolini2017weighted}, so we do not treat them again here.
Case $D_n$ is considered in Section \ref{sec:Dn}, case $\tilde B_n$ in Section \ref{sec:tBn}, case $\smash{\tilde D_n}$ in Section \ref{sec:tDn}, and case $I_2(m)$ in Section \ref{sec:I2}.
In all remaining exceptional cases we construct precise matchings via a computer program, as discussed in Section \ref{sec:exceptional-cases}.

We provide a software library to construct precise matchings for any given finite or affine Artin group, following \cite{paolini2017weighted} and the present paper. 
Source code and instructions can be found online \cite{precise-matchings}.
This library can be used to check preciseness and compute the homology.

\begin{remark}
  Not all Artin groups admit precise matchings for every cyclotomic polynomial.
  For example, consider the Coxeter system $(\W, S)$ defined by the following Coxeter matrix:
  \[
  \left(
  \begin{matrix}
    1 & 3 & 3 & 2 & \infty & 4 \\
    3 & 1 & 3 & 4 & 2 & \infty \\
    3 & 3 & 1 & \infty & 4 & 2 \\
    2 & 4 & \infty & 1 & \infty & \infty \\
    \infty & 2 & 4 & \infty & 1 & \infty \\
    4 & \infty & 2 & \infty & \infty & 1
  \end{matrix}
  \right).
  \]
  The simplicial complex $K_\W$ consists of:
  three $2$-simplices ($\{1,2,4\}, \{2,3,5\},\allowbreak \{1,3,6\}$), all having $\varphi_2$-weight equal to $3$;
  six $1$-simplices with $\varphi_2$-weight equal to $2$ ($\{1,4\},\allowbreak \{2,4\},\allowbreak \{2, 5\}, \{3, 5\}, \{1, 6\}, \{3, 6\}$), and three $1$-simplices with $\varphi_2$-weight equal to 1 ($\{1,2\}, \{2,3\}, \{1,3\}$);
  six $0$-simplices ($\{1\}, \{2\}, \{3\}, \{4\}, \{5\}, \{6\}$), all having $\varphi_2$-weight equal to $1$;
  one empty simplex, with $\varphi_2$-weight equal to 0.
  A $\varphi_2$-weighted matching can only contain edges between simplices of weight $1$.
  Since the three 1-simplices of weight 1 form a cycle, at least one of them (say $\{1,2\}$) is critical.
  Then the incidence number between $\{1,2,4\}$ and $\{1,2\}$ is non-zero, and their $\varphi_2$-weights differ by $2$.
  Therefore the matching cannot be $\varphi_2$-precise.
\end{remark}

We introduce here a few notations that will be used later.
Given a simplex $\sigma \in K_\W$, denote by $\Gamma(\sigma)$ the subgraph of $\Gamma$ induced by $\sigma$.
We will sometimes speak about the connected components of $\Gamma(\sigma)$, which will be denoted by $\Gamma_i(\sigma)$ for some index $i$.
Also, given a vertex $v\in S$, define
\[ \sigma \xor v =
  \begin{cases}
    \sigma \cup \{v\} & \text{if } v \not \in \sigma, \\
    \sigma \setminus \{v\} & \text{if } v \in \sigma.
  \end{cases}
\]

\section{Reduction to the irreducible cases}
\label{sec:reduction-to-irreducible}

Let $(\W_1, S_1)$ and $(\W_2, S_2)$ be Coxeter system, and consider the product Coxeter system $(\W_1\times \W_2, S_1 \sqcup S_2)$.
Suppose that the Artin groups $\GWn{1}$ and $\GWn{2}$ admit $\varphi$-precise matchings $\M_1$ and $\M_2$, respectively.
In the following lemma we construct a $\varphi$-precise matching for $\GWn{1} \times \GWn{2} = \G{W_1 \times W_2}$.

\begin{lemma}
  Fix a cyclotomic polynomial $\varphi$.
  Let $\M_1$ and $\M_2$ be $\varphi$-precise matchings on $\GWn{1}$ and $\GWn{2}$, respectively.
  Then $\G{W_1 \times W_2}$ also admits a $\varphi$-precise matching.
\end{lemma}

\begin{proof}
  
  First notice that the simplicial complex $K_{\W_1\times \W_2}$ consists in the simplicial join $K_{\W_1} * K_{\W_2}$ of the two simplicial complexes $K_{\W_1}$ and $K_{\W_2}$:
  \[ K_{\W_1\times \W_2} = K_{\W_1} * K_{\W_2} = \{ \sigma_1 \sqcup \sigma_2 \mid \sigma_1 \in K_{\W_1}, \; \sigma_2 \in K_{\W_2} \}. \]
  The weights behave well with respect to this decomposition:
  \[ v_\varphi(\sigma_1 \sqcup \sigma_2) = v_\varphi(\sigma_1) + v_\varphi(\sigma_2). \]
  
  Construct a matching $\M$ on $K_{\W_1\times \W_2}$ as follows.
  \begin{IEEEeqnarray*}{rCl}
    \M &=& \Big\{ \sigma_1 \sqcup \sigma_2 \to \sigma_1 \sqcup \tau_2 \;\Big|\; (\sigma_2 \to \tau_2) \in \M_2 \Big\} \; \cup \\
    && \Big\{ \sigma_1 \sqcup \sigma_2 \to \tau_1 \sqcup \sigma_2 \;\Big|\; (\sigma_1 \to \tau_1) \in \M_2 \text{ and } \sigma_2 \text{ is critical in } K_{\W_2} \Big\}.
  \end{IEEEeqnarray*}
  The critical simplices $\sigma_1\sqcup \sigma_2$ of $K_{\W_1\times \W_2}$ are those for which $\sigma_1$ is critical in $K_{\W_1}$ and $\sigma_2$ is critical in $K_{\W_2}$.
  
  Any alternating path in $K_{\W_1\times \W_2}$ projects onto an alternating path in $K_{\W_2}$ via the map $\sigma_1 \sqcup \sigma_2 \mapsto \sigma_2$ (provided that multiple consecutive occurrences of the same simplex are replaced by a single occurrence).
  This is because an edge of the form $\sigma_1 \sqcup \sigma_2 \to \sigma_1 \sqcup \tau_2$ is in $\M$ if and only if $\sigma_2 \to \tau_2$ is in $\M_2$.
  The same statement is not true for $K_{\W_1}$, but we still have a weaker property which will be useful later:
  the projection of an alternating path to $K_{\W_1}$ cannot have two consecutive edges both traversed ``upwards'' (i.e.\ increasing dimension).
  This is because if an edge of the form $\sigma_1 \sqcup \sigma_2 \to \tau_1 \sqcup \sigma_2$ is in $\M$, then $\sigma_1 \to \tau_1$ is in $\M_1$.
  
  Let us prove that $\M$ is acyclic.
  Consider an alternating cycle $c$ in $K_{\W_1\times \W_2}$.
  Its projection onto $K_{\W_2}$ gives an alternating cycle, which must be trivial because $\M_2$ is acyclic.
  Therefore $c$ takes the following form, for some fixed simplex $\sigma_2 \in K_{\W_2}$:
  \[ \sigma_{1,0} \sqcup \sigma_2 \,\rhd\, \tau_{1,1} \sqcup \sigma_2 \,\lhd\, \sigma_{1,1} \sqcup \sigma_2 \,\rhd\,
  \dots
  \,\rhd\, \tau_{1,{m}} \sqcup \sigma_2 \,\lhd\, \sigma_{1,0} \sqcup \sigma_2. \]
  If $\sigma_2$ is critical in $K_{\W_2}$, then also the projection of $c$ onto $K_{\W_1}$ is an alternating cycle. By acyclicity of $\M_1$ such a projection must be the trivial cycle, so also $c$ is trivial.
  On the other hand, if $\sigma_2$ is not critical, then none of the edges $\sigma_{1,i} \sqcup \sigma_2 \to \tau_{1,i} \sqcup \sigma_2$ is in $\M$, thus $c$ must be trivial as well.
  
  By construction, and by additivity of the weight function $v_\varphi$, the matching $\M$ is $\varphi$-weighted.
  
  Finally, suppose that $[\sigma_1 \sqcup \sigma_2 : \tau_1 \sqcup \tau_2]^\M \neq 0$, where $\sigma_1 \sqcup \sigma_2$ and $\tau_1 \sqcup \tau_2$ are critical simplices of $K_{\W_1\times \W_2}$ with $\dim(\sigma_1 \sqcup \sigma_2) = \dim(\tau_1 \sqcup \tau_2) + 1$.
  Let $\mathcal{P} \neq \varnothing$ be the set of alternating paths from $\sigma_1 \sqcup \sigma_2$ to $\tau_1 \sqcup \tau_2$. Given any path $p \in \mathcal{P}$, its projection onto $K_{\W_2}$ is an alternating path from $\sigma_2$ to $\tau_2$.
  \begin{enumerate}
    \item Suppose $\sigma_2 = \tau_2$.
    Then the projected paths are trivial in $K_{\W_2}$, so $\mathcal{P}$ is in bijection with the set of alternating paths from $\sigma_1$ to $\tau_1$ in $K_{\W_1}$.
    Therefore $[\sigma_1 : \tau_1]^{\M_1} = \pm [\sigma_1 \sqcup \sigma_2 : \tau_1 \sqcup \tau_2]^\M \neq 0$. Since $\M_1$ is $\varphi$-precise, we conclude that $v_\varphi(\sigma_1) = v_\varphi(\tau_1) + 1$ and so that $v_\varphi(\sigma_1 \sqcup \sigma_2) = v_\varphi(\tau_1 \sqcup \sigma_2) + 1$.
    \label{li:first-case}
    
    \item Suppose $\sigma_2 \neq \tau_2$.
    The projection of any $p\in \mathcal{P}$ onto $K_{\W_2}$ is a non-trivial alternating path, and $\mathcal{P}$ is non-empty, so $\dim(\sigma_2) = \dim(\tau_2) + 1$.
    Then $\dim(\sigma_1) = \dim(\tau_1)$.
    For any alternating path $p \in \mathcal{P}$, consider now its projection $q$ onto $K_{\W_1}$. We want to prove that $q$ is a trivial path (thus in particular $\sigma_1 = \tau_1$).
    Suppose by contradiction that $q$ is non-trivial. Then, since $\sigma_1$ and $\tau_1$ have the same dimension, one of the following three possibilities must occur.
    \begin{itemize}
      \item The path $q$ begins with an upward edge $\sigma_1 \lhd \rho$.
      Then $(\rho \to \sigma_1) \in \M_1$, which is not possible because $\sigma_1$ is critical.
      \item The path $q$ ends with an upward edge $\rho \lhd \tau_1$.
      Then $(\tau_1 \to \rho) \in \M_1$, which is not possible because $\tau_1$ is critical.
      \item The path $q$ begins and ends with a downward edge, so it must have two consecutive upward edges somewhere in the middle.
      This is also not possible by previous considerations.
    \end{itemize}
    We proved that the projection on $K_{\W_1}$ of any alternating path $p\in\mathcal{P}$ is trivial, and thus in particular $\sigma_1 = \tau_1$ (because $\mathcal{P}$ is non-empty).
    Then $\mathcal{P}$ is in bijection with the set of alternating paths from $\sigma_2$ to $\tau_2$ in $K_{\W_2}$.
    We conclude as in case \eqref{li:first-case}. \qedhere
  \end{enumerate}
\end{proof}

In view of this lemma, from now on we only consider irreducible Coxeter systems.

\section{Weight of irreducible components}
\label{sec:weights}

In order to compute the weight $v_\varphi(\sigma)$ of a simplex $\sigma \in K_\W$, one needs to know the Poincaré polynomial of the parabolic subgroup $\W_{\!\sigma}$ of $\W$.
Let $\Gamma_1(\sigma), \dots, \Gamma_m(\sigma)$ be the connected components of the subgraph $\Gamma(\sigma)\subseteq \Gamma$ induced by $\sigma$.
Then the Poincaré polynomial of $\W_{\!\sigma}$ splits as a product of the Poincaré polynomials of irreducible components of finite type:
\[ \W_{\!\sigma}(q) = \prod_{i=1}^m \W_{\Gamma_i(\sigma)}(q), \;\; \text{and therefore} \;\; v_\varphi(\sigma) = \sum_{i=1}^m v_\varphi(\Gamma_i(\sigma)). \]
In this section we derive formulas for the $\varphi$-weight of an irreducible component of type $A_n$, $B_n$ and $D_n$ (see Figure \ref{fig:ABD}).

\begin{figure}[htbp]
  \begin{tikzpicture}
\begin{scope}[every node/.style={circle,thick,draw,inner sep=2.5}, every label/.style={rectangle,draw=none}]
  \node (0) at (0.0,0) [label={above,minimum height=13}:$ $] {};
  \node (1) at (1.4,0) [label={above,minimum height=13}:$ $] {};
  \node (2) at (2.8,0) [label={above,minimum height=13}:$ $] {};
  \node (3) at (4.2,0) [label={above,minimum height=13}:$ $] {};
  \node (4) at (5.6,0) [label={above,minimum height=13}:$ $] {};
  \node (5) at (7.0,0) [label={above,minimum height=13}:$ $] {};
\end{scope}
\begin{scope}[every edge/.style={draw=black!60,line width=1.2}]
  \path (0) edge node {} (1);
  \path (1) edge node {} (2);
  \path (2) edge node {} (3);
  \path (3) edge node [fill=white, rectangle, inner sep=3.0, minimum height = 0.5cm] {$\ldots$} (4);
  \path (4) edge node {} (5);
\end{scope}
\begin{scope}[every edge/.style={draw=black, line width=3}]
\end{scope}
\begin{scope}[every node/.style={draw,inner sep=11.5,yshift=-4}, every label/.style={rectangle,draw=none,inner sep=6.0}, every fit/.append style=text badly centered]
\end{scope}
\end{tikzpicture}
   
  \vskip0.3cm
  
  \begin{tikzpicture}
\begin{scope}[every node/.style={circle,thick,draw,inner sep=2.5}, every label/.style={rectangle,draw=none}]
  \node (0) at (0.0,0) [label={above,minimum height=13}:$ $] {};
  \node (1) at (1.4,0) [label={above,minimum height=13}:$ $] {};
  \node (2) at (2.8,0) [label={above,minimum height=13}:$ $] {};
  \node (3) at (4.2,0) [label={above,minimum height=13}:$ $] {};
  \node (4) at (5.6,0) [label={above,minimum height=13}:$ $] {};
  \node (5) at (7.0,0) [label={above,minimum height=13}:$ $] {};
\end{scope}
\begin{scope}[every edge/.style={draw=black!60,line width=1.2}]
  \path (0) edge node {} node[inner sep=3, above] {\bf\small 4} (1);
  \path (1) edge node {} (2);
  \path (2) edge node {} (3);
  \path (3) edge node [fill=white, rectangle, inner sep=3.0, minimum height = 0.5cm] {$\ldots$} (4);
  \path (4) edge node {} (5);
\end{scope}
\begin{scope}[every edge/.style={draw=black, line width=3}]
\end{scope}
\begin{scope}[every node/.style={draw,inner sep=11.5,yshift=-4}, every label/.style={rectangle,draw=none,inner sep=6.0}, every fit/.append style=text badly centered]
\end{scope}
\end{tikzpicture}
   
  \vskip0.3cm
  
  \begin{tikzpicture}
\begin{scope}[every node/.style={circle,thick,draw,inner sep=2.5}, every label/.style={rectangle,draw=none}]
  \node (0) at (135.0:1.4) [label={above,minimum height=13}:$ $] {};
  \node (1) at (225.0:1.4) [label={above,minimum height=13}:$ $] {};
  \node (2) at (0.0,0) [label={above,minimum height=13}:$ $] {};
  \node (3) at (1.4,0) [label={above,minimum height=13}:$ $] {};
  \node (4) at (2.8,0) [label={above,minimum height=13}:$ $] {};
  \node (5) at (4.2,0) [label={above,minimum height=13}:$ $] {};
  \node (6) at (5.6,0) [label={above,minimum height=13}:$ $] {};
\end{scope}
\begin{scope}[every edge/.style={draw=black!60,line width=1.2}]
  \path (0) edge node {} (2);
  \path (1) edge node {} (2);
  \path (2) edge node {} (3);
  \path (3) edge node {} (4);
  \path (4) edge node [fill=white, rectangle, inner sep=3.0, minimum height = 0.5cm] {$\ldots$} (5);
  \path (5) edge node {} (6);
\end{scope}
\begin{scope}[every edge/.style={draw=black, line width=3}]
\end{scope}
\begin{scope}[every node/.style={draw,inner sep=11.5,yshift=-4}, every label/.style={rectangle,draw=none,inner sep=6.0}, every fit/.append style=text badly centered]
\end{scope}
\end{tikzpicture}
   \caption{Coxeter graphs of type $A_n$, $B_n$ and $D_n$. All these graphs have $n$ vertices.}
  \label{fig:ABD}
\end{figure}
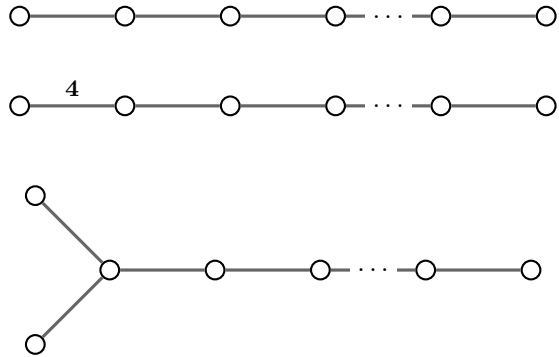

\subsection*{Components of type $A_n$}
The exponents of a Coxeter group $\W_{\! A_n}$ of type $A_n$ are $1,2,\dots, n$. Then its Poincaré polynomial is $\W_{\! A_n}(q) = [n+1]_q!$. If $\varphi_d$ is the $d$-th cyclotomic polynomial (for $d\geq 2$), the $\varphi_d$-weight is then
\[ \omega_{\varphi_d}(A_n) = \textstyle\left\lfloor \frac{n+1}{d} \right\rfloor. \]

\subsection*{Components of type $B_n$}
In this case the exponents are $1,3,\dots,2n-3,2n-1$, and the Poincaré polynomial is $\W_{\! B_n}(q) = [2n]_q!!$. The $\varphi_d$-weight (for $d \geq 2$) is given by
\[ \omega_{\varphi_d}(B_n) =
  \begin{cases}
    \left\lfloor \frac{n}{d} \right\rfloor & \text{if $d$ is odd}; \\[0.2cm]
    \left\lfloor \frac{n}{d/2} \right\rfloor & \text{if $d$ is even}.
  \end{cases} \]

\subsection*{Components of type $D_n$}
Here the exponents are $1,3,\dots,2n-3,n-1$, and the Poincaré polynomial is $\W_{\! D_n}(q) = [2n-2]_q!! \cdot [n]_q$. The $\varphi_d$-weight (for $d \geq 2$) is given by
\[ \omega_{\varphi_d}(D_n) =
  \begin{cases}
    \left\lfloor \frac{n}{d} \right\rfloor & \text{if $d$ is odd}; \\[0.2cm]
    \left\lfloor \frac{n-1}{d/2} \right\rfloor & \text{if $d$ is even and $d\nmid n$}; \\[0.2cm]
    \frac{n}{d/2} & \text{if $d$ is even and $d \mid n$}.
  \end{cases} \]

\section{Case \texorpdfstring{$A_n$}{An} revisited}
\label{sec:An}

The construction of a precise matching for the case $A_n$ was thoroughly discussed in \cite{paolini2017weighted}.
However, in order to better describe precise matchings for the cases $D_n$, $\tilde B_n$, and $\tilde D_n$, we need a slightly more general construction.

Throughout this section, let $(\W_{\! A_n}, S)$ be a Coxeter system of type $A_n$ with generating set $S=\{1, 2, \dots, n\}$, and let $K_n^A = K_{\W_{\! A_n}}$.
See Figure \ref{fig:An} for a drawing of the corresponding Coxeter graph.

\begin{figure}[htbp]
  \begin{tikzpicture}
\begin{scope}[every node/.style={circle,thick,draw,inner sep=2.5}, every label/.style={rectangle,draw=none}]
  \node (1) at (0.0,0) [label={above,minimum height=13}:$1$] {};
  \node (2) at (1.4,0) [label={above,minimum height=13}:$2$] {};
  \node (3) at (2.8,0) [label={above,minimum height=13}:$3$] {};
  \node (4) at (4.2,0) [label={above,minimum height=13}:$4$] {};
  \node (n-1) at (5.6,0) [label={above,minimum height=13}:$n-1$] {};
  \node (n) at (7.0,0) [label={above,minimum height=13}:$n$] {};
\end{scope}
\begin{scope}[every edge/.style={draw=black!60,line width=1.2}]
  \path (1) edge node {} (2);
  \path (2) edge node {} (3);
  \path (3) edge node {} (4);
  \path (4) edge node [fill=white, rectangle, inner sep=3.0, minimum height = 0.5cm] {$\ldots$} (n-1);
  \path (n-1) edge node {} (n);
\end{scope}
\begin{scope}[every edge/.style={draw=black, line width=3}]
\end{scope}
\begin{scope}[every node/.style={draw,inner sep=11.5,yshift=-4}, every label/.style={rectangle,draw=none,inner sep=6.0}, every fit/.append style=text badly centered]
\end{scope}
\end{tikzpicture}
   \caption{A Coxeter graph of type $A_n$.}
  \label{fig:An}
\end{figure}
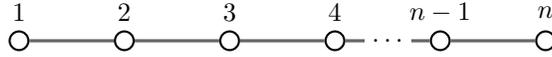

For integers $f,g \geq 0$, define $K_{n,f,g}^A \subseteq K_n$ as follows:
\[ K_{n,f,g}^A = \{ \sigma \in K_n \mid \{1,2,\dots,f\} \subseteq \sigma \text{ and } \{ n-g+1,n-g+2, \dots, n \} \subseteq \sigma \}. \]
In other words, $\smash{K_{n,f,g}^A}$ is the subset of $K_n^A$ consisting of the simplices which contain the first $f$ vertices and the last $g$ vertices.
In general, $\smash{K_{n,f,g}^A}$ is not a subcomplex of $K_n^A$.
For any $d \geq 2$, we are going to recursively construct a $\varphi_d$-weighted acyclic matching on $K_{n,f,g}^A$.
This matching coincides with the one of \cite[Section 5.1]{paolini2017weighted} when $g=0$ and $f\leq d-1$.
See also Table \ref{table:An-matching-example} for an example.

\begin{breakmatching}[$\varphi_d$-matching on $K_{n,f,g}^A$]
  \begin{enumerate}[(a)]
    \item If $f+g \geq n$ then $K_{n,f,g}^A$ has size at most $1$, and the matching is empty.
    In the subsequent cases, assume $f+g < n$.
    
    \item If $f \geq d$, then $K_{n,f,g}^A \cong K_{n-d,\,f-d,\,g}^A$ via removal of the first $d$ vertices.
    Define the matching recursively, as in $K_{n-d,\,f-d,\,g}^A$.
    In the subsequent cases, assume $f \leq d-1$.
    
    \item Case $n \geq d+g$.
    \begin{enumerate}
      \item[(c1)] If $\{1,\dots, d-1\} \subseteq \sigma$, then match $\sigma$ with $\sigma \xor d$ (here the vertex $d$ exists and can be removed, because $n\geq d+g$).
      Notice that for $f=d-1$ this is always the case, thus in the subsequent cases we can assume $f\leq d-2$.
      
      \item[(c2)] Otherwise, if $f+1 \in \sigma$ then match $\sigma$ with $\sigma \setminus\{f+1\}$.
      
      \item[(c3)] Otherwise, if $\{f+2, \dots, d-1\} \nsubseteq \sigma$ then match $\sigma$ with $\sigma \cup \{f+1\}$.
      
      \item[(c4)] We are left with the simplices $\sigma$ such that $\{1,\dots,f,f+2,\dots,d-1\}\subseteq \sigma$ and $f+1\not\in\sigma$. If we ignore the vertices $1,\dots,f+1$ we are left with the simplices on the vertex set $\{f+2,\dots,n\}$ which contain $f+2,\dots,d-1$; relabeling the vertices, these are the same as the simplices on the vertex set $\{1,\dots,n-f-1\}$ which contain $1,\dots,d-2-f$.
 Then construct the matching recursively as in $K_{n-f-1,\, d-2-f, \, g}^A$.
    \end{enumerate}
    
    \item Case $n < d+g$ (in particular, $f \leq d-2$).
    \begin{enumerate}
      \item[(d1)] If $n \equiv -1,0,1,\dots,f \pmod d$ and $\sigma$ is either
      \[ \{1,\dots n\} \text{ or } \{1,\dots,f,f+2,\dots,n\}, \]
      then $\sigma$ is critical.
      
      \item[(d2)] Otherwise, match $\sigma$ with $\sigma \xor f+1$.
    \end{enumerate}
  \end{enumerate}
  \label{matching:An}
\end{breakmatching}

\begin{table}[htbp]

{\renewcommand{\arraystretch}{1.4}

\begin{tabular}{rcl@{\hspace{15pt}}|c}

\multicolumn{3}{c|}{Simplices} & $v_\varphi(\sigma)$ \\

\hline

\begin{tikzpicture}
\begin{scope}[every node/.style={circle,thick,draw,inner sep=2.5}, every label/.style={rectangle,draw=none}]
  \node (1) at (0.0,0) [label={above,minimum height=13}:$ $,fill=black!50] {};
  \node (2) at (0.7,0) [label={above,minimum height=13}:$ $,fill=black!50] {};
  \node (3) at (1.4,0) [label={above,minimum height=13}:$ $,fill=black!50] {};
  \node (4) at (2.1,0) [label={above,minimum height=13}:$ $,fill=black!50] {};
  \node (5) at (2.8,0) [label={above,minimum height=13}:$ $,fill=black!50] {};
  \node (6) at (3.5,0) [label={above,minimum height=13}:$ $,fill=black!50] {};
  \node (7) at (4.2,0) [label={above,minimum height=13}:$ $,fill=black!50] {};
\end{scope}
\begin{scope}[every edge/.style={draw=black!60,line width=1.2}]
\end{scope}
\begin{scope}[every edge/.style={draw=black, line width=3}]
  \path (1) edge node {} (2);
  \path (2) edge node {} (3);
  \path (3) edge node {} (4);
  \path (4) edge node {} (5);
  \path (5) edge node {} (6);
  \path (6) edge node {} (7);
\end{scope}
\begin{scope}[every node/.style={draw,inner sep=11.5,yshift=-4}, every label/.style={rectangle,draw=none,inner sep=6.0}, every fit/.append style=text badly centered]
\end{scope}
\end{tikzpicture}
 & $\longrightarrow$ & 
\begin{tikzpicture}
\begin{scope}[every node/.style={circle,thick,draw,inner sep=2.5}, every label/.style={rectangle,draw=none}]
  \node (1) at (0.0,0) [label={above,minimum height=13}:$ $,fill=black!50] {};
  \node (2) at (0.7,0) [label={above,minimum height=13}:$ $,fill=black!50] {};
  \node (3) at (1.4,0) [label={above,minimum height=13}:$ $] {};
  \node (4) at (2.1,0) [label={above,minimum height=13}:$ $,fill=black!50] {};
  \node (5) at (2.8,0) [label={above,minimum height=13}:$ $,fill=black!50] {};
  \node (6) at (3.5,0) [label={above,minimum height=13}:$ $,fill=black!50] {};
  \node (7) at (4.2,0) [label={above,minimum height=13}:$ $,fill=black!50] {};
\end{scope}
\begin{scope}[every edge/.style={draw=black!60,line width=1.2}]
  \path (2) edge node {} (3);
  \path (3) edge node {} (4);
\end{scope}
\begin{scope}[every edge/.style={draw=black, line width=3}]
  \path (1) edge node {} (2);
  \path (4) edge node {} (5);
  \path (5) edge node {} (6);
  \path (6) edge node {} (7);
\end{scope}
\begin{scope}[every node/.style={draw,inner sep=11.5,yshift=-4}, every label/.style={rectangle,draw=none,inner sep=6.0}, every fit/.append style=text badly centered]
\end{scope}
\end{tikzpicture}
 & 2 \\

\begin{tikzpicture}
\begin{scope}[every node/.style={circle,thick,draw,inner sep=2.5}, every label/.style={rectangle,draw=none}]
  \node (1) at (0.0,0) [label={above,minimum height=13}:$ $,fill=black!50] {};
  \node (2) at (0.7,0) [label={above,minimum height=13}:$ $,fill=black!50] {};
  \node (3) at (1.4,0) [label={above,minimum height=13}:$ $,fill=black!50] {};
  \node (4) at (2.1,0) [label={above,minimum height=13}:$ $] {};
  \node (5) at (2.8,0) [label={above,minimum height=13}:$ $,fill=black!50] {};
  \node (6) at (3.5,0) [label={above,minimum height=13}:$ $,fill=black!50] {};
  \node (7) at (4.2,0) [label={above,minimum height=13}:$ $,fill=black!50] {};
\end{scope}
\begin{scope}[every edge/.style={draw=black!60,line width=1.2}]
  \path (3) edge node {} (4);
  \path (4) edge node {} (5);
\end{scope}
\begin{scope}[every edge/.style={draw=black, line width=3}]
  \path (1) edge node {} (2);
  \path (2) edge node {} (3);
  \path (5) edge node {} (6);
  \path (6) edge node {} (7);
\end{scope}
\begin{scope}[every node/.style={draw,inner sep=11.5,yshift=-4}, every label/.style={rectangle,draw=none,inner sep=6.0}, every fit/.append style=text badly centered]
\end{scope}
\end{tikzpicture}
 & $\longrightarrow$ & 
\begin{tikzpicture}
\begin{scope}[every node/.style={circle,thick,draw,inner sep=2.5}, every label/.style={rectangle,draw=none}]
  \node (1) at (0.0,0) [label={above,minimum height=13}:$ $,fill=black!50] {};
  \node (2) at (0.7,0) [label={above,minimum height=13}:$ $,fill=black!50] {};
  \node (3) at (1.4,0) [label={above,minimum height=13}:$ $] {};
  \node (4) at (2.1,0) [label={above,minimum height=13}:$ $] {};
  \node (5) at (2.8,0) [label={above,minimum height=13}:$ $,fill=black!50] {};
  \node (6) at (3.5,0) [label={above,minimum height=13}:$ $,fill=black!50] {};
  \node (7) at (4.2,0) [label={above,minimum height=13}:$ $,fill=black!50] {};
\end{scope}
\begin{scope}[every edge/.style={draw=black!60,line width=1.2}]
  \path (2) edge node {} (3);
  \path (3) edge node {} (4);
  \path (4) edge node {} (5);
\end{scope}
\begin{scope}[every edge/.style={draw=black, line width=3}]
  \path (1) edge node {} (2);
  \path (5) edge node {} (6);
  \path (6) edge node {} (7);
\end{scope}
\begin{scope}[every node/.style={draw,inner sep=11.5,yshift=-4}, every label/.style={rectangle,draw=none,inner sep=6.0}, every fit/.append style=text badly centered]
\end{scope}
\end{tikzpicture}
 & 2 \\

\begin{tikzpicture}
\begin{scope}[every node/.style={circle,thick,draw,inner sep=2.5}, every label/.style={rectangle,draw=none}]
  \node (1) at (0.0,0) [label={above,minimum height=13}:$ $,fill=black!50] {};
  \node (2) at (0.7,0) [label={above,minimum height=13}:$ $] {};
  \node (3) at (1.4,0) [label={above,minimum height=13}:$ $,fill=black!50] {};
  \node (4) at (2.1,0) [label={above,minimum height=13}:$ $] {};
  \node (5) at (2.8,0) [label={above,minimum height=13}:$ $,fill=black!50] {};
  \node (6) at (3.5,0) [label={above,minimum height=13}:$ $,fill=black!50] {};
  \node (7) at (4.2,0) [label={above,minimum height=13}:$ $,fill=black!50] {};
\end{scope}
\begin{scope}[every edge/.style={draw=black!60,line width=1.2}]
  \path (1) edge node {} (2);
  \path (2) edge node {} (3);
  \path (3) edge node {} (4);
  \path (4) edge node {} (5);
\end{scope}
\begin{scope}[every edge/.style={draw=black, line width=3}]
  \path (5) edge node {} (6);
  \path (6) edge node {} (7);
\end{scope}
\begin{scope}[every node/.style={draw,inner sep=11.5,yshift=-4}, every label/.style={rectangle,draw=none,inner sep=6.0}, every fit/.append style=text badly centered]
\end{scope}
\end{tikzpicture}
 & $\longrightarrow$ & 
\begin{tikzpicture}
\begin{scope}[every node/.style={circle,thick,draw,inner sep=2.5}, every label/.style={rectangle,draw=none}]
  \node (1) at (0.0,0) [label={above,minimum height=13}:$ $,fill=black!50] {};
  \node (2) at (0.7,0) [label={above,minimum height=13}:$ $] {};
  \node (3) at (1.4,0) [label={above,minimum height=13}:$ $] {};
  \node (4) at (2.1,0) [label={above,minimum height=13}:$ $] {};
  \node (5) at (2.8,0) [label={above,minimum height=13}:$ $,fill=black!50] {};
  \node (6) at (3.5,0) [label={above,minimum height=13}:$ $,fill=black!50] {};
  \node (7) at (4.2,0) [label={above,minimum height=13}:$ $,fill=black!50] {};
\end{scope}
\begin{scope}[every edge/.style={draw=black!60,line width=1.2}]
  \path (1) edge node {} (2);
  \path (2) edge node {} (3);
  \path (3) edge node {} (4);
  \path (4) edge node {} (5);
\end{scope}
\begin{scope}[every edge/.style={draw=black, line width=3}]
  \path (5) edge node {} (6);
  \path (6) edge node {} (7);
\end{scope}
\begin{scope}[every node/.style={draw,inner sep=11.5,yshift=-4}, every label/.style={rectangle,draw=none,inner sep=6.0}, every fit/.append style=text badly centered]
\end{scope}
\end{tikzpicture}
 & 1 \\

\begin{tikzpicture}
\begin{scope}[every node/.style={circle,thick,draw,inner sep=2.5}, every label/.style={rectangle,draw=none}]
  \node (1) at (0.0,0) [label={above,minimum height=13}:$ $,fill=black!50] {};
  \node (2) at (0.7,0) [label={above,minimum height=13}:$ $] {};
  \node (3) at (1.4,0) [label={above,minimum height=13}:$ $,fill=black!50] {};
  \node (4) at (2.1,0) [label={above,minimum height=13}:$ $,fill=black!50] {};
  \node (5) at (2.8,0) [label={above,minimum height=13}:$ $,fill=black!50] {};
  \node (6) at (3.5,0) [label={above,minimum height=13}:$ $,fill=black!50] {};
  \node (7) at (4.2,0) [label={above,minimum height=13}:$ $,fill=black!50] {};
\end{scope}
\begin{scope}[every edge/.style={draw=black!60,line width=1.2}]
  \path (1) edge node {} (2);
  \path (2) edge node {} (3);
\end{scope}
\begin{scope}[every edge/.style={draw=black, line width=3}]
  \path (3) edge node {} (4);
  \path (4) edge node {} (5);
  \path (5) edge node {} (6);
  \path (6) edge node {} (7);
\end{scope}
\begin{scope}[every node/.style={draw,inner sep=11.5,yshift=-4}, every label/.style={rectangle,draw=none,inner sep=6.0}, every fit/.append style=text badly centered]
\end{scope}
\end{tikzpicture}
 & & \emph{(critical)} & 2 \\

\begin{tikzpicture}
\begin{scope}[every node/.style={circle,thick,draw,inner sep=2.5}, every label/.style={rectangle,draw=none}]
  \node (1) at (0.0,0) [label={above,minimum height=13}:$ $,fill=black!50] {};
  \node (2) at (0.7,0) [label={above,minimum height=13}:$ $] {};
  \node (3) at (1.4,0) [label={above,minimum height=13}:$ $] {};
  \node (4) at (2.1,0) [label={above,minimum height=13}:$ $,fill=black!50] {};
  \node (5) at (2.8,0) [label={above,minimum height=13}:$ $,fill=black!50] {};
  \node (6) at (3.5,0) [label={above,minimum height=13}:$ $,fill=black!50] {};
  \node (7) at (4.2,0) [label={above,minimum height=13}:$ $,fill=black!50] {};
\end{scope}
\begin{scope}[every edge/.style={draw=black!60,line width=1.2}]
  \path (1) edge node {} (2);
  \path (2) edge node {} (3);
  \path (3) edge node {} (4);
\end{scope}
\begin{scope}[every edge/.style={draw=black, line width=3}]
  \path (4) edge node {} (5);
  \path (5) edge node {} (6);
  \path (6) edge node {} (7);
\end{scope}
\begin{scope}[every node/.style={draw,inner sep=11.5,yshift=-4}, every label/.style={rectangle,draw=none,inner sep=6.0}, every fit/.append style=text badly centered]
\end{scope}
\end{tikzpicture}
 & & \emph{(critical)} & 1

\end{tabular}

}

\vskip0.3cm

   \caption{Matching \ref{matching:An} on $K^A_{7,1,3}$ for $d=3$.}
  \label{table:An-matching-example}
\end{table}

\begin{lemma}
  Matching \ref{matching:An} is acyclic.
  \label{lemma:An-acyclic}
\end{lemma}

\begin{proof}
  The proof is by induction on $n$, the case $n=0$ being trivial.
  In case (a) the matching is empty and so it is acyclic.
  In case (b), the matching on $\smash{K_{n-d,\, f-d,\, g}^A}$ is acyclic by induction, and therefore also the matching on $K^A_{n,f,g}$ is acyclic.
  
  Consider now case (c).
  Let $\eta\colon K^A_{n,f,g} \to \{ p_1 > p_4 > p_{2,3} \}$ be the poset map that sends $\sigma$ to: $p_1$, if subcase (c1) applies; $p_{2,3}$, if subcase (c2) or (c3) applies; $p_4$, if subcase (c4) applies.
  We have that $\eta(\sigma) = \eta(\tau)$ whenever $\sigma$ is matched with $\tau$.
  In addition, on each fiber $\eta^{-1}(p)$ the matching is acyclic (for $p=p_4$, this follows by induction).
  Therefore by the Patchwork Theorem \cite[Theorem 11.10]{kozlov2008combinatorial} the entire matching is acyclic.
  
  We are left with case (d). Here the matching is of the form $\{\sigma \longleftrightarrow \sigma \xor (f+1)\}$ (possibly leaving out a pair), so it is acyclic.
\end{proof}

\begin{lemma}
  Matching \ref{matching:An} is $\varphi_d$-weighted.
  \label{lemma:An-weighted}
\end{lemma}

\begin{proof}
  This proof is also by induction on $n$.
  In case (b), removing the first $d$ vertices decreases all $\varphi_d$-weights by $1$, so the matching is $\varphi_d$-weighted by induction.
  
  Consider now case (c).
  In (c1) we have to check that $\omega_{\varphi_d}(A_{d-1} \sqcup A_k) = \omega_{\varphi_d}(A_{d+k})$ for all $k\geq 0$.
  This follows immediately from the formula for $\varphi_d$-weights (Section \ref{sec:weights}).
  In (c2)-(c3) we match simplices by adding or removing $f+1$, which alters the size of the leftmost connected component. However this component has always size $\leq d-2$, so it does not contribute to the $\varphi_d$-weight.
  Finally, in (c4) the matching is $\varphi_d$-weighted by induction.
  
  In case (d) most of the vertices belong to the rightmost connected component, which has size $\geq g$.
  If we add the vertex $f+1$ to a simplex $\sigma \in K^A_{n,f,g}$ not containing $f+1$, we either create a leftmost connected component of size $\leq d-2$, or we join the leftmost component with the rightmost component creating the full simplex $\{1,2,\dots,n\}$.
  In the former case, the leftmost component is small enough not to contribute to the $\varphi_d$-weight.
  In the latter case we have $\sigma = \{1,\dots, f,f+2,\dots,n\}$; it is easy to check that $\omega_{\varphi_d}(\sigma) = \omega_{\varphi_d}(\sigma \cup \{f+1\})$ if and only if $n\equiv f+1,\dots,d-2 \pmod d$.
\end{proof}

\begin{lemma}
  The critical simplices of Matching \ref{matching:An} are given by Table \ref{table:An-critical}.
  In particular, the matching is $\varphi_d$-precise for $f=g=0$.
\label{lemma:matching-An}
\end{lemma}

\begin{table}[htbp]
  \begin{center}
    {
    \tabulinesep=4pt
    \begin{tabu}{c|c|c|c}
      \multicolumn{2}{c|}{Case} & \# Critical & $|\sigma| - v_\varphi(\sigma)$ \\
      \hline \hline
      \multicolumn{2}{c|}{$f > n$ or $g > n$} & $0$ & - \\
      \hline
      \multicolumn{2}{c|}{$f,g \leq n$ and $f+g \geq n$} & $1$ & $n - \lfloor \frac{n+1}{d} \rfloor$ \\
      \hline
      \parbox[t]{4mm}{\multirow{3}{*}{\rotatebox[origin=c]{90}{$f+g < n$}}} &
      $\begin{array}{l}
        n \equiv \max(d-1, f_0 + g_0 + 1), \dots, \\
        \min(f_0+d-1,g_0+d-1) \mod d
        \end{array}$ &
      \multirow{2}{*}{$2$} &
      $n - \lfloor\frac{n-f}{d}\rfloor - \lfloor\frac{n-g}{d}\rfloor - 1$ \\
      \cline{2-2}\cline{4-4}
      & $\begin{array}{l}
        n \equiv \max(f_0, g_0), \dots, \\
        \min(f_0+g_0,d-2) \mod d
        \end{array}$ & &
      $n - \lfloor\frac{n-f}{d}\rfloor - \lfloor\frac{n-g}{d}\rfloor$ \\
      \cline{2-4}
      & else & $0$ & - \\
      \hline
    \end{tabu}}
  \end{center}
  \vskip0.3cm
  \caption{Critical simplices of Matching \ref{matching:An}.\\
  Here $f_0, g_0\in \{0,\dots, d-1\}$ are defined as $f$ mod $d$, and $g$ mod $d$, respectively.
  Moreover, the notation ``$n\equiv a,\dots, b \mod d$'' means that $n$ is conguent modulo $d$ to some integer in the closed interval $[a,b]$.}
  \label{table:An-critical}
\end{table}

\begin{remark}
  \ \\[-0.4cm]
  \begin{itemize}
   \item The conditions in Table \ref{table:An-critical} are symmetric in $f$ and $g$, even if the definition of Matching \ref{matching:An} is not.
   
   \item The two intervals of Table \ref{table:An-critical} in the case $f+g<n$ are always disjoint.

   \item For $g=0$, Matching \ref{matching:An} coincides with the matching defined in \cite[Section 5.1]{paolini2017weighted}.
   Table \ref{table:An-critical} simplifies a lot in this case, as both intervals contain only one element ($d-1$ and $f_0$, respectively).
   See \cite[Table 2]{paolini2017weighted}.
   
   \item If $f\equiv -1 \pmod d$ or $g\equiv -1 \pmod d$, then the two intervals of Table \ref{table:An-critical} are empty, thus there is at most 1 critical simplex.
   
   \item If $\sigma \to \tau$ is in the matching, then $\sigma = \tau \cup \{v \}$ with $v \equiv 0$ or $v \equiv f+1 \pmod d$.
   This can be easily checked by induction.
  \end{itemize}
\end{remark}

\begin{proof}[Proof of Lemma \ref{lemma:matching-An}]
  As a preliminary step, rewrite the two intervals of Table \ref{table:An-critical} as follows.
  \begin{itemize}
    \item The condition $n \equiv \max(d-1, f_0 + g_0 + 1), \dots, \min(f_0+d-1,g_0+d-1) \mod d$ is equivalent to:
    \begin{equation}
      \begin{cases} n \equiv -1,0,\dots,f_0-1 \pmod d, \\
      n-g \equiv f_0+1,\dots,d-1 \pmod d.
      \end{cases}
      \label{eqn:An-first-interval}
    \end{equation}
    
    \item The condition $n \equiv \max(f_0, g_0), \dots, \min(f_0+g_0,d-2) \mod d$ is equivalent to:
    \begin{equation}
      \begin{cases} n \equiv f_0,\dots,d-2 \pmod d, \\
	n-g \equiv 0,\dots,f_0 \pmod d.
      \end{cases}
      \label{eqn:An-second-interval}
    \end{equation}
  \end{itemize}
  Throughout the proof, we will refer to these two conditions as ``Condition \eqref{eqn:An-first-interval} (resp.\ \eqref{eqn:An-second-interval}) for $K_{n,f,g}^A$''.
  
  We prove the lemma by induction on $n$, the case $n=0$ being trivial.
  If $f+g \geq n$ there is nothing to prove, thus we can assume $f+g < n$.
  
  Suppose to be in case (b) of Matching \ref{matching:An}, i.e.\ $f\geq d$.
  Conditions \eqref{eqn:An-first-interval} and \eqref{eqn:An-second-interval} for $K_{n,f,g}^A$ are equivalent to those for $K_{n-d,\,f-d,\,g}^A$.
  The critical simplices of $K_{n,f,g}^A$ are in one-to-one correspondence with the critical simplices of $K_{n-d,\,f-d,\,g}^A$ via removal of the first $d$ vertices.
  This correspondence decreases the size of a simplex by $d$, and decreases the $\varphi_d$-weight by $1$.
  Adding $d-1$ to the values of the last column of Table \ref{table:An-critical} for $K_{n-d,\,f-d,\,g}^A$, we exactly recover Table \ref{table:An-critical} for $K_{n,f,g}^A$.
  
  Suppose to be in case (c), i.e.\ $f\leq d-1$ and $n \geq d+g$.
  Critical simplices only arise from subcase (c4).
  Notice that $(d-2-f) + g < (n-f-1)$, so by induction the critical simplices of $K_{n-f-1,\,d-2-f,\,g}^A$ are described by the last three rows of Table \ref{table:An-critical}.
  Condition \eqref{eqn:An-first-interval} (resp.\ \eqref{eqn:An-second-interval}) for $K_{n-f-1,\,d-2-f,\,g}^A$ coincides with Condition \eqref{eqn:An-second-interval} (resp. \eqref{eqn:An-first-interval}) for $K_{n,f,g}^A$.
  The critical simplices of $K_{n,f,g}^A$ are in one-to-one correspondence with the critical simplices of $K_{n-f-1,\,d-2-f,\,g}^A$ via removal of the first $f$ vertices.
  This correspondence decreases the size by $f$, and leaves the $\varphi_d$-weight unchanged.
  Adding $f$ to the values of the last column of Table \ref{table:An-critical} for $K_{n-f-1,\,d-2-f,\,g}^A$, we recover Table \ref{table:An-critical} for $K_{n,f,g}^A$.
  
  Suppose to be in case (d), i.e., $f \leq d-2$ and $n < d+g$. Since $f+g < n$, we must have $n-g \in \{f+1, \dots, d-1\}$. In particular Condition \eqref{eqn:An-second-interval} cannot hold, and the first part of Condition \eqref{eqn:An-first-interval} holds.
  Case (d1) happens if and only if $n\equiv -1,0,\dots,f \pmod d$, i.e., if and only if the second part of Condition \eqref{eqn:An-first-interval} holds.
  If this happens then there are two critical simplices, namely $\{1,\dots,n\}$ and $\{1,\dots,f,f+2,\dots,n\}$. The difference $|\sigma| - v_\varphi(\sigma)$ is the same for these two simplices, and is given by $n - \lfloor\frac{n+1}{d}\rfloor = (n-1) - \lfloor\frac{n-f}{d}\rfloor$. Since $\lfloor\frac{n-g}{d}\rfloor = 0$, we can rewrite it also as $n-\lfloor\frac{n-f}{d}\rfloor - \lfloor\frac{n-g}{d}\rfloor-1$.
\end{proof}

\section{Case \texorpdfstring{$D_n$}{Dn}}
\label{sec:Dn}

In \cite{paolini2017weighted} precise matchings were constructed for $A_n$ and $B_n$, but the third infinite family of groups of finite type, namely $D_n$, was left out (see Figure \ref{fig:Dn}).

\begin{figure}[htbp]
  \begin{tikzpicture}
\begin{scope}[every node/.style={circle,thick,draw,inner sep=2.5}, every label/.style={rectangle,draw=none}]
  \node (1) at (135.0:1.4) [label={above,minimum height=13}:$1$] {};
  \node (2) at (225.0:1.4) [label={above,minimum height=13}:$2$] {};
  \node (3) at (0.0,0) [label={above,minimum height=13}:$3$] {};
  \node (4) at (1.4,0) [label={above,minimum height=13}:$4$] {};
  \node (5) at (2.8,0) [label={above,minimum height=13}:$5$] {};
  \node (n-1) at (4.2,0) [label={above,minimum height=13}:$n-1$] {};
  \node (n) at (5.6,0) [label={above,minimum height=13}:$n$] {};
\end{scope}
\begin{scope}[every edge/.style={draw=black!60,line width=1.2}]
  \path (1) edge node {} (3);
  \path (2) edge node {} (3);
  \path (3) edge node {} (4);
  \path (4) edge node {} (5);
  \path (5) edge node [fill=white, rectangle, inner sep=3.0, minimum height = 0.5cm] {$\ldots$} (n-1);
  \path (n-1) edge node {} (n);
\end{scope}
\begin{scope}[every edge/.style={draw=black, line width=3}]
\end{scope}
\begin{scope}[every node/.style={draw,inner sep=11.5,yshift=-4}, every label/.style={rectangle,draw=none,inner sep=6.0}, every fit/.append style=text badly centered]
\end{scope}
\end{tikzpicture}
   \caption{A Coxeter graph of type $D_n$.}
  \label{fig:Dn}
\end{figure}
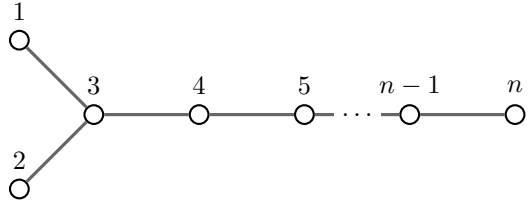

For $n\geq 4$ let $(\W_{\! D_n}, S)$ be a Coxeter system of type $D_n$, with generating set $S = \{1,2,\dots,n\}$, and let $K_n^D = K_{\W_{\! D_n}}$.
We are going to construct a $\varphi_d$-precise matching on $K_n^D$.
We actually split the definition according to the parity of $d$, and for $d$ even we construct a matching on each
\[ K_{n,g}^D = \{ \sigma \in K_n^D \mid \{ n-g+1, n-g+2. \dots, n\} \subseteq \sigma \}, \]
for $0 \leq g \leq n-1$. We will need this construction for $\tilde B_n$ and $\tilde D_n$.

\begin{breakmatching}[$\varphi_d$-matching on $K_{n}^D$ for $d$ odd]
  \begin{enumerate}[(a)]
    \item If $1 \in \sigma$ then match $\sigma$ with $\sigma \xor 2$.
    \item Otherwise, relabel the vertices $\{2,\dots, n\}$ as $\{1,\dots, n-1\}$ and construct the matching as in $K_{n-1}^A$.
  \end{enumerate}
  \label{matching:Dn-odd}
\end{breakmatching}

\begin{breakmatching}[$\varphi_d$-matching on $K_{n,g}^D$ for $d$ even]
  \begin{enumerate}[(a)]
    \item If $2\not\in \sigma$, relabel the vertices $\{1,3,4,\dots,n\}$ as $\{1,\dots,n-1\}$ and construct the matching as in $K_{n-1,\,0,\,g}^A$.
    
    \item Otherwise, if $d=2$ and $\{1,2,3,4\} \nsubseteq \sigma$, proceed as follows.
    \begin{enumerate}
      \item[(b1)] If $\{1,2,4\} \subseteq \sigma$: match $\sigma$ with $\sigma \xor 5$ if possible (i.e.\ if $n-g \geq 5$); else $\sigma$ is critical.
      \item[(b2)] Otherwise: match $\sigma$ with $\sigma \xor 3$ if possible (i.e.\ if $n-g \geq 3$); else $\sigma$ is critical.
    \end{enumerate}
    
    \item Otherwise, if $d\geq 4$ and $3\not\in\sigma$, match $\sigma$ with $\sigma\xor 1$.
    
    \item Otherwise, if $d=4$ and $4\not\in\sigma$ (recall that at this point $\{2,3\}\subseteq \sigma$), ignore vertex $1$, relabel vertices $\{5,\dots,n\}$ as $\{1,\dots,n-4\}$, and construct the matching as in $K_{n-4, \, 0, \, g}^A$.
    
    \item Otherwise, if $d\geq 6$ and $4\not\in\sigma$, match $\sigma$ with $\sigma \xor 1$.
    
    \item Otherwise, if $d\geq 4$ and $1\not\in\sigma$, proceed as follows. Recall that at this point $\{2,3,4\} \subseteq \sigma$.
    \begin{enumerate}
      \item[(f1)] If $\{2,\dots,\frac d2+1\} \subseteq \sigma$, relabel the vertices $\{2,\dots, n\}$ as $\{ 1,\dots, n-1\}$ and construct the matching as in $K_{n-1,\, \max(\frac d2, \, 3), \, g}^A$.
      
      \item[(f2)] Otherwise, match $\sigma$ with $\sigma \cup \{1\}$.
    \end{enumerate}
    
    \item Otherwise, proceed as follows. Recall that at this point $\{1,2,3,4\} \subseteq \sigma$.
    Let $k \geq 4$ be the size of the connected component $\Gamma_1(\sigma)$ of the vertex $1$, in the subgraph $\Gamma(\sigma)\subseteq \Gamma$ induced by $\sigma$.
    Write $k = q\frac d2 + r$ where:
    \[ \begin{cases}
      0 < r < \frac d2 & \text{if $k\not\equiv 0 \pmod {\frac d2}$}; \\
      r \in \{0, \frac d2\} \text{ and $q$ even} & \text{if $k\equiv 0 \pmod {\frac d2}$}.
    \end{cases} \]
    Define a vertex $v$ as follows:
    \[ v = \begin{cases}
            q\frac d2 + 1 & \text{if $q$ is even}; \\
            q\frac d2 + 2 & \text{if $q$ is odd}.
           \end{cases} \]
    It can be checked that $v=1$ or $v\geq 5$.
    The idea now is that most of the times $\sigma \xor v$ has the same $\varphi_d$-weight as $\sigma$.
    Unfortunately there are some exceptions, so we still have to examine a few subcases.
    
    \begin{enumerate}
      \item[(g1)] Suppose $v\in\sigma$.
      If $v \leq n-g$, match $\sigma$ with $\sigma \setminus \{v \}$.
      Otherwise $\sigma$ is critical.
      
      \item[(g2)] Suppose $v\not\in\sigma$. Match $\sigma$ with $\sigma \cup \{v\}$, unless one of the following occurs.
      \begin{enumerate}
       \item[(g2.1)] $v > n$ (i.e.\ the vertex $v$ doesn't exist in $S$). Then $\sigma$ is critical.
       
       \item[(g2.2)] $q$ is even, and $\{q\frac d2 + 2, \dots, (q+1)\frac d2 + 1\} \subseteq \sigma$. In this case the connected components $\Gamma_1(\sigma)$ and $\Gamma_1(\sigma\cup\{v\})$ have a different $\varphi_d$-weight.
       Then ignore the vertices $\leq q\frac d2 + 1$, relabel the vertices $\{q\frac d2 + 2, \dots, n\}$ as $\{1, \dots, n-q\frac d2 - 1\}$ and construct the matching as in $K_{n-q\frac d2 - 1, \, \frac d2, \, g}^A$.
       
       \item[(g2.3)] $q$ is odd, and $\{q\frac d2 + 3, \dots, (q+1)\frac d2\} \subseteq \sigma$. Similarly to case (g2.2), relabel the vertices and construct the matching as in $K_{n-q\frac d2 - 2, \, \frac d2 - 2, \, g}^A$.
      \end{enumerate}

    \end{enumerate}

  \end{enumerate}
  \label{matching:Dn-even}
\end{breakmatching}

\begin{lemma}
  Matchings \ref{matching:Dn-odd} and \ref{matching:Dn-even} are acyclic and $\varphi_d$-weighted.
  Critical simplices for these matchings on $K_n^D$ are given by Tables \ref{table:Dn-critical-odd} and \ref{table:Dn-critical-even}.
  In particular, both matchings on $K_n^D$ are $\varphi_d$-precise.
  \label{lemma:matching-Dn}
\end{lemma}

\begin{table}[htbp]
  \begin{center}
    {
    \tabulinesep=4pt
    \begin{tabu}{c|c|c}
      Case & \# Critical & $|\sigma| - v_\varphi(\sigma)$ \\
      \hline \hline
      
      $n\equiv 0 \pmod d$ & \multirow{2}{*}{$2$} & $n - 2\,\frac{n}{d}$ \\
      \cline{1-1}\cline{3-3}
      $n \equiv 1 \pmod d$ & & $n - 2\,\frac{n-1}{d} - 1$ \\
      \cline{1-3}
      else & $0$ & - \\
      \hline
    \end{tabu}}
  \end{center}
  \vskip0.3cm
  \caption{Critical simplices of Matching \ref{matching:Dn-odd} (case $D_n$, $d$ odd).}
  \label{table:Dn-critical-odd}
\end{table}

\begin{table}[htbp]
  \begin{center}
    {
    \tabulinesep=4pt
    \begin{tabu}{c|c|c}
      Case & Origin & $|\sigma| - v_\varphi(\sigma)$ \\
      \hline \hline
      
      $n\equiv 0 \pmod d$ &
      \begin{minipage}{0.33\textwidth}
        (a), (b) for $n=4$ and $d=2$, (d) for $d=4$, (f) for $d\geq 6$ or $n=d=4$, (g2.1), (g2.2), (g2.3)
      \end{minipage}
      & $n - 2 \frac nd$ 
      \\
      \hline
      
      $n\equiv 1 \pmod d$ &
      \begin{minipage}{0.33\textwidth}
        (a)
      \end{minipage}
      & $n - 2 \frac {n-1}{d} - 1$ \\
      \hline
      
      $n\equiv \frac d2 + 1 \pmod d$ for $d\geq 4$ &
      \begin{minipage}{0.33\textwidth}
        (d) for $d=4$, (f) for $d\geq 6$, (g2.1), (g2.2), (g2.3)
      \end{minipage}
      & $n - 2 \frac {n-1}{d}$ \\
      \hline
      
      else & - & - \\
      \hline
    \end{tabu}}
  \end{center}
  \vskip0.3cm
  \caption{Critical simplices of Matching \ref{matching:Dn-even} (case $D_n$, $d$ even) for $g=0$.
  In the second column we indicate in which parts of Matching \ref{matching:Dn-even} the critical simplices arise.}
  \label{table:Dn-critical-even}
\end{table}

\begin{proof}[Sketch of proof]
  The proof is similar to those of Lemmas \ref{lemma:An-acyclic} and \ref{lemma:An-weighted}.
  For every $d \geq 2$, the quantity $|\sigma| - v_{\varphi_d}(\sigma)$ is constant among the critical simplices.
  Therefore the matchings are $\varphi_d$-precise.
\end{proof}

\section{Case \texorpdfstring{$\tilde B_n$}{tilde Bn}}
\label{sec:tBn}

Consider now, for $n \geq 3$, an affine Coxeter system $(\W_{\!\tilde B_n}, S)$ of type $\tilde B_n$ (see Figure \ref{fig:tBn}).
Throughout this section, let $\smash{K_n = K_{\W_{\!\tilde B_n}}}$.
We are going to describe a $\varphi_d$-precise matching on $K_n$.
For $d$ odd the matching is very simple, and has exactly one critical simplex. For $d$ even the situation is more complicated.

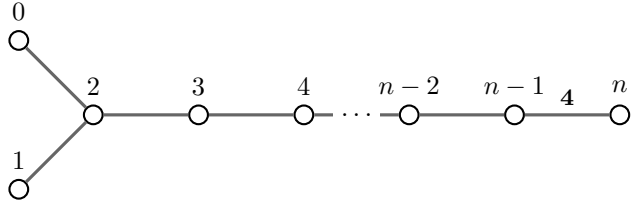
\begin{figure}[htbp]
  \begin{tikzpicture}
\begin{scope}[every node/.style={circle,thick,draw,inner sep=2.5}, every label/.style={rectangle,draw=none}]
  \node (0) at (135.0:1.4) [label={above,minimum height=13}:$0$] {};
  \node (1) at (225.0:1.4) [label={above,minimum height=13}:$1$] {};
  \node (2) at (0.0,0) [label={above,minimum height=13}:$2$] {};
  \node (3) at (1.4,0) [label={above,minimum height=13}:$3$] {};
  \node (4) at (2.8,0) [label={above,minimum height=13}:$4$] {};
  \node (n-2) at (4.2,0) [label={above,minimum height=13}:$n-2$] {};
  \node (n-1) at (5.6,0) [label={above,minimum height=13}:$n-1$] {};
  \node (n) at (7.0,0) [label={above,minimum height=13}:$n$] {};
\end{scope}
\begin{scope}[every edge/.style={draw=black!60,line width=1.2}]
  \path (0) edge node {} (2);
  \path (1) edge node {} (2);
  \path (2) edge node {} (3);
  \path (3) edge node {} (4);
  \path (4) edge node [fill=white, rectangle, inner sep=3.0, minimum height = 0.5cm] {$\ldots$} (n-2);
  \path (n-2) edge node {} (n-1);
  \path (n-1) edge node {} node[inner sep=3, above] {\bf\small 4} (n);
\end{scope}
\begin{scope}[every edge/.style={draw=black, line width=3}]
\end{scope}
\begin{scope}[every node/.style={draw,inner sep=11.5,yshift=-4}, every label/.style={rectangle,draw=none,inner sep=6.0}, every fit/.append style=text badly centered]
\end{scope}
\end{tikzpicture}
   \caption{A Coxeter graph of type $\tilde B_n$.}
  \label{fig:tBn}
\end{figure}

\begin{matching}[$\varphi_d$-matching on $K_n = K_{\W_{\!\tilde B_n}}$ for $d$ odd]
  For $\sigma \neq \{1,2,\dots, n\}$, match $\sigma$ with $\sigma \xor 0$.
  Then $\{1,2,\dots, n\}$ is the only critical simplex.
  \label{matching:tBn-odd}
\end{matching}

\begin{matching}[$\varphi_d$-matching on $K_n = K_{\W_{\!\tilde B_n}}$ for $d$ even]
  For $\sigma \in K_n$, let $k$ be the size of the connected component $\Gamma_n(\sigma)$ of the vertex $n$, in the subgraph $\Gamma(\sigma) \subseteq \Gamma$ induced by $\sigma$.
  Let $k = q \frac d2 + r$, with $0\leq r < \frac d2$.
  \begin{enumerate}[(a)]
    \item If $r \geq 1$ match $\sigma$ with $\sigma \xor \big(n - q\frac d2\big)$, unless $\sigma = \{0,2,3,\dots,n\}$ and $r=1$ (in this case $\sigma$ is critical).
    
    \item If $r = 0$ and $\big\{ n - (q+1)\frac d2 + 1, \dots, n - q\frac d2 - 1 \big\} \subseteq \sigma$: ignore vertices $\geq n - q\frac d2$, relabel vertices $\{0,1,\dots, n-q\frac d2 - 1\}$ as $\{1,2, \dots, n - q\frac d2\}$, and construct the matching as in $K^D_{n-q\frac d2, \, \frac d2 - 1}$.
    
    \item If $r = 0$ and $\big\{ n - (q+1)\frac d2 + 1, \dots, n - q\frac d2 - 1 \big\} \nsubseteq \sigma$, proceed as follows.
    \begin{enumerate}
      \item[(c1)] If $|\sigma|=n$ (i.e.\ $\sigma$ is either $\{0,2,3,\dots, n\}$ or $\{1,2,3,\dots, n\}$), then $\sigma$ is critical.
      \item[(c2)] If $n = (q+1)\frac d2$ and $\sigma = \{0,2,3, \dots, n-q\frac d2 - 1, n-q\frac d2 + 1, \dots, n\}$, then $\sigma$ is critical.
      \item[(c3)] Otherwise, match $\sigma$ with $\sigma \xor \big( n - q\frac d2 \big)$.
    \end{enumerate}
  \end{enumerate}
  \label{matching:tBn-even}
\end{matching}

\begin{lemma}
  Matchings \ref{matching:tBn-odd} and \ref{matching:tBn-even} are acyclic and $\varphi_d$-weighted.
  For $d$ odd, Matching \ref{matching:tBn-odd} has exactly one critical simplex $\sigma$ which satisfies $|\sigma| - v_{\varphi_d}(\sigma) = n - \big\lfloor \frac nd \big\rfloor$.
  For $d$ even, all critical simplices $\sigma$ of Matching \ref{matching:tBn-even} satisfy $|\sigma| - v_{\varphi_d}(\sigma) = n - \big\lfloor \frac{n}{d/2} \big\rfloor$.
  In particular, both matchings are $\varphi_d$-precise.
  \label{lemma:matching-tBn}
\end{lemma}

\begin{proof}[Sketch of proof]
  The first part can be proved along the lines of Lemmas \ref{lemma:An-acyclic} and \ref{lemma:An-weighted}.
  The formulas for $|\sigma| - v_{\varphi_d}(\sigma)$ can be checked examining all critical simplices.
  Finally, since the quantity $|\sigma| - v_{\varphi_d}(\sigma)$ is constant among the critical simplices, both matchings are $\varphi_d$-precise.
\end{proof}

\section{Case \texorpdfstring{$\tilde D_n$}{tilde Dn}}
\label{sec:tDn}

In this section we consider a Coxeter system $(\W_{\!\tilde D_n}, S)$ of type $\tilde D_n$, for $n \geq 4$ (see Figure \ref{fig:tDn}).
Throughout this section, let $\smash{K_n = K_{\W_{\!\tilde D_n}}}$.
We are going to describe a $\varphi_d$-precise matching on $K_n$.
Again, this will be easier for $d$ odd and quite involved for $d$ even.

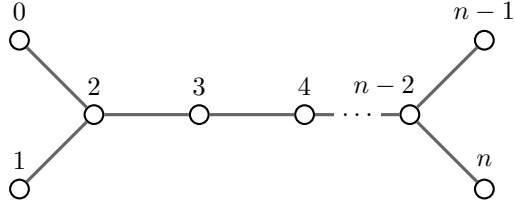
\begin{figure}[htbp]
  \begin{tikzpicture}
\begin{scope}[every node/.style={circle,thick,draw,inner sep=2.5}, every label/.style={rectangle,draw=none}]
  \node (0) at (135.0:1.4) [label={above,minimum height=13}:$0$] {};
  \node (1) at (225.0:1.4) [label={above,minimum height=13}:$1$] {};
  \node (2) at (0.0,0) [label={above,minimum height=13}:$2$] {};
  \node (3) at (1.4,0) [label={above,minimum height=13}:$3$] {};
  \node (4) at (2.8,0) [label={above,minimum height=13}:$4$] {};
  \node (n-2) at (4.2,0) [label={93,minimum height=13}:$n-2\!\!\!$] {};
  \node (n-1) at ($(4.2,0) + (45.0:1.4)$) [label={above,minimum height=13}:$n-1$] {};
  \node (n) at ($(4.2,0) + (-45.0:1.4)$) [label={above,minimum height=13}:$n$] {};
\end{scope}
\begin{scope}[every edge/.style={draw=black!60,line width=1.2}]
  \path (0) edge node {} (2);
  \path (1) edge node {} (2);
  \path (2) edge node {} (3);
  \path (3) edge node {} (4);
  \path (4) edge node [fill=white, rectangle, inner sep=3.0, minimum height = 0.5cm] {$\ldots$} (n-2);
  \path (n-2) edge node {} (n-1);
  \path (n-2) edge node {} (n);
\end{scope}
\begin{scope}[every edge/.style={draw=black, line width=3}]
\end{scope}
\begin{scope}[every node/.style={draw,inner sep=11.5,yshift=-4}, every label/.style={rectangle,draw=none,inner sep=6.0}, every fit/.append style=text badly centered]
\end{scope}
\end{tikzpicture}
   \caption{A Coxeter graph of type $\tilde D_n$.}
  \label{fig:tDn}
\end{figure}

\begin{breakmatching}[$\varphi_d$-matching on $K_n = K_{\W_{\!\tilde D_n}}$ for $d$ odd]
  \begin{enumerate}[(a)]
    \item If $\sigma = \{1,2,\dots, n\}$, then $\sigma$ is critical.
    \item If $1 \not\in \sigma$, relabel vertices $\{n, n-1, \dots, 3,2,0\}$ as $\{1, 2, \dots, n\}$ and generate the matching as in $K^D_n$.
    \item In all remaining cases, match $\sigma$ with $\sigma \xor 0$.
  \end{enumerate}

  \label{matching:tDn-odd}
\end{breakmatching}

\begin{matching}[$\varphi_d$-matching on $K_n = K_{\W_{\!\tilde D_n}}$ for $d$ even]
  For $n=4$ and $d\leq 6$, we construct the matching separately as follows.
  \begin{itemize}
   \item Case $n=4$, $d=2$.
   If $|\sigma|=1$ and $2\not\in\sigma$, or $|\sigma| = 2$, or $|\sigma|=3$ and $2\in\sigma$, then match $\sigma$ with $\sigma \xor 2$.
   Otherwise, $\sigma$ is critical.
   
   \item Case $n=4$, $d=4$.
   If $2\not\in\sigma$ or $\sigma \cap \{1,3,4\} = \varnothing$, then match $\sigma$ with $\sigma \xor 0$.
   Otherwise, $\sigma$ is critical.
   
   \item Case $n=4$, $d=6$.
   Match $\sigma$ with $\sigma \xor 0$, except in the following two cases: $2\in\sigma$, $0\not\in\sigma$ and $|\sigma| \geq 3$; or, $\{0,2\}\subseteq\sigma$ and $|\sigma| = 4$.
  \end{itemize}
  In the remaining cases ($n\geq 5$ or $d \geq 8$), the matching is constructed as follows.
  \begin{enumerate}[(a)]
    \item If $1\not\in\sigma$, relabel vertices $\{n, n-1, \dots, 3,2,0\}$ as $\{1,2,\dots,n\}$ and construct the matching as in $K^D_n$.
    
    \item Otherwise, if $d=2$ and $\{0,1,2,3\} \nsubseteq \sigma$, proceed as follows.
    \begin{enumerate}
      \item[(b1)] If $\{0,1,3\}\subseteq\sigma$, match $\sigma$ with $\sigma \xor 4$ if $\{5,6,\dots,n\} \nsubseteq \sigma$, else $\sigma$ is critical.
      \item[(b2)] Otherwise, if $\{1,3,4,\dots,n\} \nsubseteq \sigma$ then match $\sigma$ with $\sigma\xor 2$, else $\sigma$ is critical.
    \end{enumerate}
    
    \item Otherwise, if $d\geq 4$ and $0\not\in\sigma$, proceed as follows.
    \begin{enumerate}
      \item[(c1)] If $\{1,2,\dots,\frac d2\} \subseteq \sigma$, relabel vertices $\{n, n-1, \dots, 2,1\}$ as $\{1,2,\dots, n\}$ and construct the matching as in $K^D_{n, \, \frac d2}$.
      \item[(c2)] Otherwise, if $n = \frac d2 + 1$ and $\sigma = \{1,2,\dots, n-2, n\}$, then $\sigma$ is critical.
      \item[(c3)] Otherwise, if $\sigma = \{1,2,\dots, n\}$, then $\sigma$ is critical.
      \item[(c4)] Otherwise, match $\sigma$ with $\sigma \cup \{0\}$.
    \end{enumerate}
    
    \item Otherwise, if $d \geq 4$ and $2\not\in\sigma$, match $\sigma$ with $\sigma \xor 0$.
    
    \item Otherwise, if $d = 4$ and $3\not\in\sigma$, ignore vertices $0,1,2$, relabel vertices $\{n,\allowbreak n-1,\allowbreak \dots,\allowbreak 4\}$ as $\{1,2,\dots,n-3\}$ and construct the matching as in $K^D_{n-3}$.
    
    \item Otherwise, if $d\geq 6$ and $3\not\in\sigma$, match $\sigma$ with $\sigma \xor 0$.
    
    \item Otherwise, proceed as follows.
    Recall that at this point $\{0,1,2,3\} \subseteq \sigma$.
    Let $k\geq 4$ be the size of the leftmost connected component $\Gamma_0(\sigma)$ of the subgraph $\Gamma(\sigma) \subseteq \Gamma$ induced by $\sigma$.
    Notice that $\{0,1,\dots,k-1\} \subseteq \sigma$, unless $k=n$ and $\sigma = \{0,1,\dots, n-2, n\}$.
    Similarly to Matching \ref{matching:Dn-even}, write $k = q\frac d2 + r$ where:
    \[ \begin{cases}
      0 < r < \frac d2 & \text{if $k\not\equiv 0 \pmod {\frac d2}$}; \\
      r \in \{0, \frac d2\} \text{ and $q$ even} & \text{if $k\equiv 0 \pmod {\frac d2}$}.
    \end{cases} \]
    Define a vertex $v$ as follows:
    \[ v = \begin{cases}
            q\frac d2 & \text{if $q$ is even}; \\
            q\frac d2 + 1 & \text{if $q$ is odd}.
           \end{cases} \]
    
    \begin{enumerate}
      \item[(g1)] If $d=4$, $q$ odd and $r=1$, proceed as follows.
      \begin{enumerate}
	\item[(g1.1)] If $k\leq n-2$, ignore vertices $0,\dots,k$, relabel vertices $\{n,\allowbreak n-1,\allowbreak \dots,\allowbreak k+1\}$ as $\{1,2,\dots,n-k\}$, and construct the matching as in $K^D_{n-k}$.
	
	\item[(g1.2)] Otherwise, $\sigma$ is critical.
      \end{enumerate}
      
      \item[(g2)] Otherwise, if $\sigma = \{0,1,\dots, n-2, n\}$ and $v\geq n-1$, then $\sigma$ is critical.
      
      \item[(g3)] Otherwise, if $v\in\sigma$, match $\sigma$ with $\sigma \xor v$.
      
      \item[(g4)] Otherwise, if $v > n$, then $\sigma$ is critical.
      
      \item[(g5)] Otherwise, proceed as follows.
      Let $c$ be the size of the (possibly empty) connected component $C = \Gamma_{v+1}(\sigma)$ of the vertex $v+1$, in the subgraph $\Gamma(\sigma) \subseteq \Gamma$ induced by $\sigma$.
      Let
      \[ \ell = \begin{cases}
		  \frac d2 & \text{if $q$ even}; \\
		  \frac d2 - 2 & \text{if $q$ odd}.
                \end{cases} \]
      \begin{enumerate}
	\item[(g5.1)] If $\{n-1, n\} \subseteq C$, then $\sigma$ is critical.
	
	\item[(g5.2)] Otherwise, if $c < \ell$, match $\sigma$ with $\sigma \cup \{v\}$.
	
	\item[(g5.3)] Otherwise, if $c = \ell$, $n-1\not\in C$ and $n \in C$, then $\sigma$ is critical.
	
	\item[(g5.4)] Otherwise, ignore vertices $0,1,\dots,v-1$, relabel vertices $\{n,n-1,\dots,v+1\}$ as $\{1,2,\dots, n-v\}$ and construct the matching as in $K^D_{n-v, \,\ell}$.
      \end{enumerate}
    \end{enumerate}
  \end{enumerate}

  \label{matching:tDn-even}
\end{matching}

\begin{lemma}
  Matchings \ref{matching:tDn-odd} and \ref{matching:tDn-even} are $\varphi_d$-precise.
  In addition, critical simplices of Matching \ref{matching:tDn-odd} are as in Table \ref{table:tDn-critical-odd}.
\end{lemma}

\begin{table}[htbp]
  \begin{center}
    {
    \tabulinesep=4pt
    \begin{tabu}{c|c|c}
      Case & \# Critical & $|\sigma| - v_\varphi(\sigma)$ \\
      \hline \hline
      
      $n\equiv 0 \pmod d$ & \multirow{2}{*}{$3$} & $n - \frac nd$ (once), \, $n - 2\,\frac{n}{d}$ (twice) \\
      \cline{1-1}\cline{3-3}
      $n \equiv 1 \pmod d$ & & $n - \frac{n-1}{d}$ (once), \, $n - 2\,\frac{n-1}{d} - 1$ (twice) \\
      \cline{1-3}
      else & $1$ & $n - \lfloor \frac nd \rfloor$ \\
      \hline
    \end{tabu}}
  \end{center}
  \vskip0.3cm
  \caption{Critical simplices of Matching \ref{matching:tDn-odd} (case $\tilde D_n$, $d$ odd).}
  \label{table:tDn-critical-odd}
\end{table}

\begin{proof}[Sketch of proof]
  We only discuss the critical simplices of Matching \ref{matching:tDn-odd} (i.e.\ the case $d$ odd), and see why this matching is precise.
  The check for Matching \ref{matching:tDn-even} is much more involved and will be omitted.
  
  Following the definition of Matching \ref{matching:tDn-odd}, one critical simplex is always given by $\bar\sigma = \{1,2,\dots, n\}$.
  It has size $|\bar\sigma| = n$ and weight $v_{\varphi}(\bar\sigma) = \left\lfloor \frac nd \right\rfloor$.
  The other critical simplices arise from the matching on $K_n^D$, and therefore from the matching on $K_{n-1}^A$. There are two of them (which we denote by $\bar\tau_1$ and $\bar\tau_2$) when $n \equiv 0,1 \pmod d$, and zero otherwise.
  
  If $n\equiv 0 \pmod d$, the sizes of $\bar\tau_1$ and $\bar\tau_2$ are $n - 2\frac nd + 1$ and $n - 2 \frac nd$, and their weights are $1$ and $0$, respectively.
  For $\frac nd \geq 2$ we have that $|\bar\tau_1| \leq |\sigma| - 2$, so the incidence number between $\bar\sigma$ and any of $\bar\tau_1$ and $\bar\tau_2$ is zero.
  This is enough to conclude that the matching is $\varphi_d$-precise, since $|\tau_1| - v_\varphi(\tau_1) = |\tau_2| - v_\varphi(\tau_2)$.
  The remaining case is $\frac nd = 1$, i.e.\ $n=d$.
  Here we have $\bar \tau_1 = \{ 0,2,3\dots, n-2,n-1 \}$ and $\bar \tau_2 = \{0,2,3,\dots,n-2\}$.
  There are exactly two alternating paths from $\bar\sigma$ to $\bar\tau_1$:
  \begin{itemize}
    \item $\{ 1,2,\dots, n \} \vartriangleright \{ 1,2,\dots, n-1 \} \vartriangleleft \{ 0,1,\dots, n-1\} \vartriangleright \{ 0, 2,3, \dots, n-1 \}$;
    
    \item $\{1,2,\dots, n\} \vartriangleright \{1,2,\dots, n-2, n\} \vartriangleleft \{0,1,\dots, n-2, n\} \vartriangleright \{ 0,2,3,\dots, n-2,\allowbreak n\} \vartriangleleft \{ 0,2,3,\dots, n\} \vartriangleright \{ 0,2,3,\dots, n-1\}$.
  \end{itemize}
  These two paths give opposite contributions to the incidence number $[\bar\sigma:\bar\tau_1]^\M$, so $[\bar\sigma:\bar\tau_1]^\M = 0$.
  Therefore the matching is $\varphi_d$-precise.
  
  If $n\equiv 1 \pmod d$, the sizes of $\bar\tau_1$ and $\bar\tau_2$ are $n-2\frac{n-1}d$ and $n-2\frac{n-1}d - 1$, and their weights are $1$ and $0$, respectively.
  Then $|\bar\tau_1| \leq |\bar\sigma| - 2$, so the incidence number between $\sigma$ and any of $\bar\tau_1$ and $\bar\tau_2$ is always zero.
\end{proof}

Table \ref{table:homology-tDn} shows the local homology in the case $\smash{\tilde D_n}$ for $n \leq 9$, computed using the software library \cite{paolini2017weighted}.
We employ the notation $\{d\} = R / (\varphi_d)$, as in \cite{de1999arithmetic, de2001arithmetic}.

\begin{table}
  \begin{center}
  {\def\arraystretch{1.25}
  \begin{tabular}{c|c|c|c|c|c|c|}
      \cline{2-7}
      \rule{0pt}{13pt} & $\tilde D_4$ & $\tilde D_5$ & $\tilde D_6$ & $\tilde D_7$ & $\tilde D_8$ & $\tilde D_9$ \\
      \hline
      \multicolumn{1}{|c|}{$H_0$} & $\{2\}$ & $\{2\}$ & $\{2\}$ & $\{2\}$ & $\{2\}$ & $\{2\}$ \\
      \hline
      \multicolumn{1}{|c|}{$H_1$} & $\{3\}$ & $0$ & $0$ & $0$ & $0$ & $0$ \\
      \hline
      \multicolumn{1}{|c|}{$H_2$} & $\{4\}^3$ & $\{4\}$ & $\{3\}$ & $\{3\}$ & $0$ & $0$ \\
      \hline
      \multicolumn{1}{|c|}{$H_3$} & $m_{\tilde D_4}$ & $\{5\}$ & $\{5\}$ & $0$ & $0$ & $\{3\}$ \\
      \hline
      \multicolumn{1}{|c|}{$H_4$} & $R$ & $m_{\tilde D_5}$ & $\{4\} \oplus \{6\}^{3}$ & $\{4\}^{3} \oplus \{6\}$ & $\{4\}^{3}$ & $\{4\}$ \\
      \hline
      \multicolumn{1}{|c|}{$H_5$} & \graycell & $R$ & $m_{\tilde D_6}$ & $\{4\}^{4} \oplus \{7\}$ & $\{4\}^{2} \oplus \{7\}$ & $0$ \\
      \hline
      \multicolumn{1}{|c|}{$H_6$} & \multicolumn{2}{c|}{\graycell} & $R$ & $m_{\tilde D_7}$ & $\{4\}^{2} \oplus \{6\} \oplus \{8\}^{3}$ & $\{8\}$ \\
      \hline
      \multicolumn{1}{|c|}{$H_7$} & \multicolumn{3}{c|}{\graycell} & $R$ & $m_{\tilde D_8}$ & $\{6\} \oplus \{9\}$ \\
      \hline
      \multicolumn{1}{|c|}{$H_8$} & \multicolumn{4}{c|}{\graycell} & $R$ & $m_{\tilde D_9}$ \\
      \hline
      \multicolumn{1}{|c|}{$H_9$} & \multicolumn{5}{c|}{\graycell} & $R$ \\
      \hline
  \end{tabular}
  }
  \end{center}
  
  \begin{align*}
    m_{\tilde D_4} &= \{2\}^4 \oplus \{4\}^3 \oplus \{6\}^3 \\
    m_{\tilde D_5} &= \{2\}^{2} \oplus \{4\} \oplus \{6\} \oplus \{8\}^{3} \\
    m_{\tilde D_6} &= \{2\}^{5} \oplus \{4\}^{2} \oplus \{6\}^{3} \oplus \{8\} \oplus \{10\}^{3} \\
    m_{\tilde D_7} &= \{2\}^{3} \oplus \{4\}^{5} \oplus \{6\} \oplus \{8\} \oplus \{10\} \oplus \{12\}^{3} \\
    m_{\tilde D_8} &= \{2\}^{6} \oplus \{4\}^{4} \oplus \{6\}^{2} \oplus \{8\}^{3} \oplus \{10\} \oplus \{12\} \oplus \{14\}^{3} \\
    m_{\tilde D_9} &= \{2\}^{4} \oplus \{4\}^{2} \oplus \{6\}^{2} \oplus \{8\} \oplus \{10\} \oplus \{12\} \oplus \{14\} \oplus \{16\}^{3}
  \end{align*}
  
  \vskip0.3cm
  \caption{Homology in the case $\tilde D_n$ for $n \leq 9$.}
  \label{table:homology-tDn}
\end{table}

\section{Case \texorpdfstring{$I_2(m)$}{I2(m)}}
\label{sec:I2}

In this section we consider the case $I_2(m)$ for $m\geq 5$ (see Figure \ref{fig:I2}).
The Poincaré polynomial is given by $\W_{\! I_2(m)} = [2]_q [m]_q$.
Then the $\varphi_d$-weight is
\[ \omega_{\varphi_d}(I_2(m)) =
  \begin{cases}
    2 & \text{if } d=2 \text{ and $m$ even}; \\
    1 & \text{if } d=2 \text{ and $m$ odd}; \\
    1 & \text{if } d \geq 3 \text{ and } d\mid m; \\
    0 & \text{if } d \geq 3 \text{ and } d\nmid m.
  \end{cases}
\]

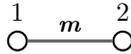
\begin{figure}[htbp]
  \begin{tikzpicture}
  \begin{scope}[every node/.style={circle,thick,draw,inner sep=2.5}, every label/.style={rectangle,draw=none}]
    \node (1) at (0.0,0) [label={above,minimum height=13}:$1$] {};
    \node (2) at (1.4,0) [label={above,minimum height=13}:$2$] {};
  \end{scope}
  \begin{scope}[every edge/.style={draw=black!60,line width=1.2}]
    \path (1) edge node {} node[inner sep=3, above] {\bf\small \textit{m}} (2);
  \end{scope}
  \end{tikzpicture}
  \caption{A Coxeter graph of type $I_2(m)$.}
  \label{fig:I2}
\end{figure}

In this case $\varphi_d$-precise matchings are easy to construct by hand.
As a straightforward consequence we also obtain the homology groups $H_*(\XW; R)$.

\begin{breakmatching}[$\varphi_d$-matching on $K_{\W_{\! I_2(m)}}$]
  \begin{itemize}
    \item If $d=2$ and $m$ is even, every simplex is critical.
    Critical simplices are then: $\{1,2\}$ (size $2$, weight $2$), $\{1\}$, $\{2\}$ (size $1$, weight $1$), and $\varnothing$ (size $0$, weight $0$).
    By Theorem \ref{thm:homology-artin-groups}, the homology groups are: $H_0(\XW;R)_{\varphi_2} \cong \frac{R}{(\varphi_2)}$, and $H_1(\XW;R)_{\varphi_2} \cong \frac{R}{(\varphi_2)}$.
    
    \item If $d=2$ and $m$ is odd, match $\{1,2\}$ with $\{1\}$ (both simplices have weight $1$).
    The critical simplices are: $\{2\}$ (size $1$, weight $1$), and $\varnothing$ (size $0$, weight $0$).
    The homology groups are: $H_0(\XW;R)_{\varphi_2} \cong \frac{R}{(\varphi_2)}$, and $H_1(\XW;R)_{\varphi_2} \cong 0$.
    
    \item If $d\geq 3$ and $d \mid m$, match $\{2\}$ with $\varnothing$ (both simplices have weight $0$).
    The critical simplices are: $\{1,2\}$ (size $2$, weight $1$), and $\{1\}$ (size $1$, weight $0$).
    The homology groups are: $H_0(\XW;R)_{\varphi_d} \cong 0$, and $H_1(\XW;R)_{\varphi_d} \cong \frac{R}{(\varphi_d)}$.
    
    \item If $d\geq 3$ and $d\nmid m$, match $\{1,2\}$ with $\{1\}$ and $\{2\}$ with $\varnothing$ (all simplices have weight $0$).
    There are no critical simplices, and all homology groups are trivial.
  \end{itemize}
  \label{matching:I2}
\end{breakmatching}

To summarize, the local homology is given by:
\[ H_0(\XW; R) \cong \frac{R}{(\varphi_2)}; \qquad H_1(\XW; R) \cong \bigoplus_{\substack{
    d \mid m\\
    d \geq 2}} \frac{R}{(\varphi_d)}. \]
This result corrects the one given in \cite{de1999arithmetic}, where proper divisors of $m$ were not taken into account.

\section{Exceptional cases}
\label{sec:exceptional-cases}

In all exceptional finite and affine cases (see for example \cite[Appendix A1]{bjorner2005combinatorics} for a classification), we constructed precise matchings by means of a computer program.
The explicit description of these matchings, together with proof of preciseness and homology computations, can be obtained through the software library available online \cite{precise-matchings}.

The homology groups can be computed using Theorem \ref{thm:homology-artin-groups}.
They are described in Tables \ref{table:homology-finite} and \ref{table:homology-affine}, where again we employ the notation $\{d\} = R / (\varphi_d)$.
We recover the results of \cite{de1999arithmetic} (for the finite cases) and \cite{salvetti2013combinatorial} (for the affine cases), except for minor corrections in the cases $E_8$ and $\smash{\tilde E_8}$.

\begin{table}
  \begin{center}
  {\def\arraystretch{1.25}
  \begin{tabular}{c|c|c|c|c|c|c|}
      \cline{2-7}
      \rule{0pt}{13pt} & $H_3$ & $H_4$ & $F_4$ & $E_6$ & $E_7$ & $E_8$ \\
      \hline
      \multicolumn{1}{|c|}{$H_0$} & $\{2\}$ & $\{2\}$ & $\{2\}$ & $\{2\}$ & $\{2\}$ & $\{2\}$ \\
      \hline
      \multicolumn{1}{|c|}{$H_1$} & $0$ & $0$ & $\{2\}$ & $0$ & $0$ & $0$ \\
      \hline
      \multicolumn{1}{|c|}{$H_2$} & $m_{H_3}$ & $0$ & $\{2\} \oplus \{3\} \oplus \{6\}$ & $0$ & $0$ & $0$  \\
      \hline
      \multicolumn{1}{|c|}{$H_3$} & \graycell & $m_{H_4}$ & $m_{F_4}$ & $0$ & $0$ & $0$ \\
      \hline
      \multicolumn{1}{|c|}{$H_4$} & \multicolumn{3}{c|}{\graycell} & $\{6\} \oplus \{8\}$ & $\{6\}$ & $\{4\}$ \\
      \hline
      \multicolumn{1}{|c|}{$H_5$} & \multicolumn{3}{c|}{\graycell} & $m_{E_6}$ & $\{6\}$ & $0$ \\
      \hline
      \multicolumn{1}{|c|}{$H_6$} & \multicolumn{4}{c|}{\graycell} & $m_{E_7}$ & $\{8\} \oplus \{12\}$ \\
      \hline
      \multicolumn{1}{|c|}{$H_7$} & \multicolumn{5}{c|}{\graycell} & $m_{E_8}$ \\
      \hline
  \end{tabular}
  }
  \end{center}
  
  \begin{align*}
    m_{H_3} &= \{2\} \oplus \{6\} \oplus \{10\} \\
    m_{H_4} &= \{2\} \oplus \{3\} \oplus \{4\} \oplus \{5\} \oplus \{6\} \oplus \{10\} \oplus \{12\} \oplus \{15\} \oplus \{20\} \oplus \{30\} \\
    m_{F_4} &= \{2\} \oplus \{3\} \oplus \{4\} \oplus \{6\} \oplus \{8\} \oplus \{12\} \\
    m_{E_6} &= \{3\} \oplus \{6\} \oplus \{9\} \oplus \{12\} \\
    m_{E_7} &= \{2\} \oplus \{6\} \oplus \{14\} \oplus \{18\} \\
    m_{E_8} &= \{2\} \oplus \{3\} \oplus \{4\} \oplus \{5\} \oplus \{6\} \oplus \{8\} \oplus \{10\} \oplus \{12\} \oplus \{15\} \oplus \{20\} \\
    &\quad\; \oplus \{24\} \oplus \{30\}
  \end{align*}
  
  \vskip0.3cm
  \caption{Homology in the exceptional finite cases.}
  \label{table:homology-finite}
\end{table}

\begin{table}
  \begin{center}
  {\def\arraystretch{1.25}
  \begin{tabular}{c|c|c|c|c|c|c|}
      \cline{2-7}
      \rule{0pt}{13pt} & $\tilde I_1$ & $\tilde G_2$ & $\tilde F_4$ & $\tilde E_6$ & $\tilde E_7$ & $\tilde E_8$ \\
      \hline
      \multicolumn{1}{|c|}{$H_0$} & $\{2\}$ & $\{2\}$ & $\{2\}$ & $\{2\}$ & $\{2\}$ & $\{2\}$ \\
      \hline
      \multicolumn{1}{|c|}{$H_1$} & $R$ & $\{2\} \oplus \{3\}$ & $\{2\}$ & $0$ & $0$ & $0$ \\
      \hline
      \multicolumn{1}{|c|}{$H_2$} & \graycell & $R$ & $\{2\} \oplus \{3\}$ & $0$ & $0$ & $0$ \\
      \hline
      \multicolumn{1}{|c|}{$H_3$} & \multicolumn{2}{c|}{\graycell} & $m_{\tilde F_4}$ & $\{3\}$ & $\{3\}$ & $0$ \\
      \hline
      \multicolumn{1}{|c|}{$H_4$} & \multicolumn{2}{c|}{\graycell} & $R$ & $\{5\} \oplus \{8\}$ & $0$ & $\{4\}$ \\
      \hline
      \multicolumn{1}{|c|}{$H_5$} & \multicolumn{3}{c|}{\graycell} & $m_{\tilde E_6}$ & $0$ & $0$ \\
      \hline
      \multicolumn{1}{|c|}{$H_6$} & \multicolumn{3}{c|}{\graycell} & $R$ & $m_{\tilde E_7}$ & $\{5\} \oplus \{8\}$ \\
      \hline
      \multicolumn{1}{|c|}{$H_7$} & \multicolumn{4}{c|}{\graycell} & $R$ & $m_{\tilde E_8}$ \\
      \hline
      \multicolumn{1}{|c|}{$H_8$} & \multicolumn{5}{c|}{\graycell} & $R$ \\
      \hline
  \end{tabular}
  }
  \end{center}
  
  \begin{align*}
    m_{\tilde F_4} &= \{2\}^{2} \oplus \{3\} \oplus \{4\} \oplus \{8\} \\
    m_{\tilde E_6} &= \{2\} \oplus \{3\}^{3} \oplus \{6\}^{2} \oplus \{9\}^{2} \oplus \{12\}^{2} \\
    m_{\tilde E_7} &= \{2\}^{3} \oplus \{3\} \oplus \{4\} \oplus \{6\} \oplus \{8\} \oplus \{10\} \oplus \{14\} \oplus \{18\} \\
    m_{\tilde E_8} &= \{2\}^{2} \oplus \{3\} \oplus \{4\} \oplus \{5\} \oplus \{8\} \oplus \{9\} \oplus \{14\}
  \end{align*}
  
  \vskip0.3cm
  \caption{Homology in the exceptional affine cases.}
  \label{table:homology-affine}
\end{table}

\section*{Acknowledgements}
I would like to thank Mario Salvetti for his strong support and help during my PhD studies.

\bibliographystyle{amsalpha-abbr}
\bibliography{bibliography}

\end{document}